\documentclass{amsart}

\usepackage[utf8x]{inputenc}

\usepackage{amssymb}
\usepackage[active]{srcltx}
\usepackage[all]{xy}
\usepackage{enumerate}
\usepackage{xspace}
\usepackage{float}
\usepackage{tikz}
\usepackage{pgf}
\usetikzlibrary{arrows}
\usetikzlibrary{positioning}
\usepackage{pgfkeys}
\usepackage{xcolor}

\newcommand{\leftexp}[2]{{\vphantom{#2}}^{#1}{#2}}

\newcommand{\numberset}[1]{\ensuremath{\mathbb{#1}}} 
\newcommand{\calm}[1]{\mathcal{#1}}
\newcommand{\N}{\numberset{N}} 
\newcommand{\R}{\numberset{R}} 
\newcommand{\Z}[1]{\numberset{Z}_{#1}} 
\newcommand{\C}{\numberset{C}} 
\newcommand{\rp}[1]{\R\numberset{P}^{#1}}

\newcommand{\Sp}[1]{\ensuremath{\mathbb{S}^{#1}}}
\newcommand{\Stres}{\ensuremath{\mathbb{S}^3}}
\newcommand{\B}[1]{\mathbb{B}^{#1}}

\newcommand{\D}{\mathbb{D}}
\newcommand{\Lfg}{L_{f\bar{g}}}
\newcommand{\LF}{L_{F}}
\newcommand{\F}{\calm{F}}

\newcommand{\Fp}{\F^{\prime}}
\newcommand{\Lp}{L^{\prime}}
\newcommand{\Os}{\calm{O}}
\newcommand{\cc}{\calm{C}}
\newcommand{\BD}[1]{\mathbf{D}^{#1}}

\newcommand{\Ho}[2]{\mathrm{H}_{#1}(#2)}

\newcommand{\norm}[1]{\lVert #1\rVert}	
\newcommand{\va}[1]{|#1|}

\newcommand{\fbg}{f\bar{g}}
\newcommand{\fbgxy}{f(x,y)\overline{g(x,y)}}
\newcommand{\fbgplusr}{f(x,y)\overline{g(x,y)}+z^r}
\newcommand{\cP}{\calm{P}}
\newcommand{\mpfg}{\Phi_{\fbg}}
\newcommand{\mpp}{\Phi^{\prime}}
\newcommand{\gp}{\gamma^{\prime}}
\newcommand{\hp}{h^{\prime}}
\newcommand{\id}{id}
\newcommand{\pr}[1]{#1^{-1}}

\newcommand{\Ng}[1]{\calm{G}(#1)}
\newcommand{\W}{\calm{W}}
\newcommand{\Op}[1]{\calm{O}(#1)}
\newcommand{\mtp}[2]{#1^{#2}}
\newcommand{\ptp}[1]{\pi^{#1}}
\newcommand{\ostp}[1]{\Os^{(#1)}}
\newcommand{\otp}[1]{m^{(#1)}}
\newcommand{\gtp}[1]{g^{(#1)}}
\newcommand{\ltp}[1]{\lambda^{(#1)}}
\newcommand{\stp}[1]{\sigma^{(#1)}}
\newcommand{\ttp}[1]{t^{(#1)}}

\newcommand{\mc}[1]{\calm{#1}}
\newcommand{\kn}[2]{(\Sp{#1},#2)}

\newcommand{\br}{[r]}
\newcommand{\vc}[1]{\overrightarrow{#1}}

\newtheorem{thm}{Theorem}[section]
\newtheorem{prop}[thm]{Proposition}
\newtheorem{lem}[thm]{Lemma}
\newtheorem{cor}[thm]{Corollary}
\theoremstyle{definition}
\newtheorem{rmk}[thm]{Remark}
\newtheorem{exa}[thm]{Example}
\newtheorem{defin}[thm]{Definition}

\newcommand{\defi}[1]{\textbf{#1}}

\newcommand{\ie}{\textit{i.e.},\xspace}

\begin{document}
\title{The topology of real suspension singularities of type $\fbg+z^n$.}
\author[H. Aguilar-Cabrera]{Hayd\'ee Aguilar-Cabrera}
\address{Department of Mathematics, Columbia University\\
MC 4406\\
2990 Broadway\\
New York, NY\\
10027}
\email{langeh@gmail.com}
\date{\today}
\thanks{Research partially supported by CONACyT grants U55084 and J49048-F (Mexico), by ECOS-ANUIES grant M06-M02 (France-Mexico) and the Laboratorio Internacional Solomon Lefschetz (CNRS-CONACyT, France and Mexico).}
\keywords{Real singularities, Milnor fibration, Graph manifolds, Open book decomposition}
\subjclass[2010]{Primary 32S25,32S55; Secondary 32S50,57M27}
\begin{abstract}
In this article we study the topology of a family of real analytic germs $F \colon (\C^3,0) \to (\C,0)$
with isolated critical point at $0$, given by $F(x,y,z)=f(x,y)\overline{g(x,y)}+z^r$, where $f$ and $g$ are holomorphic, $r \in \Z{}^+$ and $r \geq 2$. We describe the link $L_F$ as a graph manifold using its natural open book decomposition, related to the Milnor fibration of the map-germ $\fbg$ and the description of its monodromy as a quasi-periodic diffeomorphism through its Nielsen invariants. Furthermore, such a germ $F$ gives rise to a Milnor fibration $\frac{F}{|F|} \colon \Sp{5} \setminus L_F \to \Sp{1}$. We present a join theorem, which allows us to describe the homotopy type of the Milnor fibre of $F$ and we show some cases where the open book decomposition of $\Sp{5}$ given by the Milnor fibration of $F$ cannot come from the Milnor fibration of a complex singularity in $\C^3$.
\end{abstract}

\maketitle

\section{Introduction}
For some years, there has been interest in the study of the topology of suspension hypersurface singularities of type $F=f+z^r$, where $f$ is a holomorphic function from $\C^2$ to $\C$ and $z \in \C$, see for instance \cite{MR0358797,MR2134278,MR0488073,MR1824957,MR1709489,MR1877769,MR1669948}.

In the real analytic case, in \cite{MR2922705} we were interested in the suspension singularities of type $F(x,y,z)=\overline{xy}(x^p+y^q)+z^r$ from $\C^3$ to $\C$ where $p,q,r \in \Z{}^{+}$, $p,q,r \geq 2$ and $\gcd(p,q)=1$. This family of functions has isolated singularity at the origin, we proved that the link $L_F$ is a Seifert manifold and we gave its Seifert invariants. Furthermore, we presented families of this kind of singularities such that the corresponding Milnor fibration does not give an open book decomposition of $\Sp{5}$ that comes from Milnor fibrations of complex singularities.

In this paper we are interested in a more general situation: given $f,g \colon (\C^2,0) \to (\C,0)$ holomorphic germs such that the real analytic germ $\fbg$ has isolated critical point at the origin, we define a new real analytic germ $F \colon (\C^3,0) \to (\C,0)$ given by $F(x,y,z)=f(x,y)\overline{g(x,y)}+z^r$ and we describe the topology of the link $L_F$ as a graph manifold.

Among the several works studying the topology of suspension singularities, \cite{MR1709489} is particularly relevant to this work. In \cite{MR1709489}, A.~Pichon studied the topology of the link $L_F$ for $F=f+z^r$ a suspension singularity, where $f$ is a reduced holomorphic germ from $\C^2$ to $\C$. Moreover, it is proved that for a $3$-manifold $M$, there exists, at most, a finite number of holomorphic reduced germs $f \colon (\C^2,0) \to (\C,0)$ such that $M$ is homeomorphic to the link $L_F$, where $F=f+z^r$.

In the present article we generalise the method given in \cite{MR1709489} in order to study the topology of the link $L_F$ for $F(x,y,z)=f(x,y)\overline{g(x,y)}+z^r$. The description of $L_F$ as a graph manifold is given in terms of $f$, $g$ and $r$, and it is based in the existence of a cyclic branched $r$-cover $\calm{P} \colon L_F \to \Sp{3}$ with the link $L_{\fbg}$ as the ramification locus.

In \cite{PichSea:barfg}, A.~Pichon and J.~Seade proved that given $f,g \colon (\C^2,0) \to (\C,0)$ holomorphic germs, if the germ $\fbg$ has isolated critical point, then its Milnor fibration $\mpfg$ is given by $\ensuremath{\mpfg=\frac{\fbg}{|\fbg|}}$. The monodromy $h$ of this Milnor fibration is a quasi-periodic diffeomorphism. Thus, through a pull-back diagram, we equip $L_F$ with an open book decomposition with open book fibration $\mpp$, such that a representative of the corresponding monodromy is the quasi-periodic diffeomorphism $h^r$. We compute the Nielsen invariants of the monodromy $h$ and then, using results given in \cite{MR1709489}, we compute the Nielsen invariants of the monodromy $h^r$. By Theorem~\ref{mero}, we obtain the Waldhausen decomposition of $L_F$ associated to the Waldhausen decomposition of $\Sp{3}$ given by the Milnor fibration $\mpfg$.

Furthermore, we prove that $F$ is a $d$-regular function (see \cite{MR2647448}) and by \cite[Theorem~5.3]{MR2647448} we have that $F$ has Milnor fibration $\ensuremath{\Phi_F=\frac{F}{|F|}}$. 
We also prove that if $F$ is of the type $F=f+z^r$, where $f$ is a $d$-regular real analytic function, then the pair $\kn{n+1}{\LF}$ is the  $r$-fold cyclic suspension of the knot $\kn{n-1}{L_f}$. This allows us to adapt \cite[Lemma~6.1]{MR0488073} into a Join Theorem for $d$-regular functions, which gives the homotopy type of the Milnor fibre of $F$ as the join of the Milnor fibre of $f$ and $r$ points.

Finally, we give the algorithm to compute the topology of the link $L_F$ as a graph manifold from the Milnor fibration $\mpfg$ and $r$. Moreover, we apply such algorithm in some examples and, as in \cite{MR2922705}, we show that in these examples the open book decomposition of $\Sp{5}$ given by the Milnor fibration $\Phi_F$ cannot come from Milnor fibrations of complex singularities from $\C^3$ to $\C$.

\section{Preliminaries on fibred plumbing links}\label{prel}
This section presents some results from \cite{NC96}, \cite{DBM}, \cite{MR1824957} and \cite{PichSea:barfg} on fibred plumbing links and their monodromy. For this, we present the concepts of plumbing manifolds, plumbing links and their representation by a graph (see for example \cite[\S4]{MR1709489}).

\begin{defin}
A \defi{plumbing manifold} $M$ is a $3$-manifold such that $M$ is boundary of a $4$-manifold $P(\Gamma)$ obtained by plumbing according to a plumbing tree $\Gamma$.
\end{defin}

\begin{defin}
A \defi{plumbing link} is a pair $(M,L)$ where $M$ is a plumbing manifold and $L = K_1 \cup \ldots \cup K_n$ is an oriented link in $M$ which is a union (possibly empty) of $\Sp{1}$-fibres of the plumbed $\D^2$-bundles. 
\end{defin}

Notice that each $K_i$ has a natural orientation as the boundary of a $\D^2$-fibre.  We denote by $-K_i$ the knot $K_i$ endowed with the opposite orientation. Then the oriented link $L$ will be denoted by  $L = \epsilon_1 K_1 \cup \ldots \cup \epsilon_n K_n$ where $\epsilon_i \in \{-1,+1\}$.

The homeomorphism class of the pair $(M,L)$ is given by the plumbing tree $\Gamma$ decorated with arrows in the following way: for each component $K_i$ of the link $L$, we attach an arrow weighted by the multiplicity $(\epsilon_i)$ to the vertex corresponding to the $\D^2$-bundle of which $K_i$ is a $\Sp{1}$-fibre. 

\begin{defin}\label{fiblink}
A \defi{fibred plumbing link} $(M,L)$ is a plumbing link $(M,L)$ such that there exists an open book fibration $\Phi : M \setminus L \to \Sp{1}$ with binding $L$. 
\end{defin}

The following theorem is a generalisation of a result of Eisenbud and Neumann (\cite[Th~11.2]{EN85}) reformulated in terms of plumbing links. In \cite{EN85} this result is proved for multilinks in  $\Z{}$-homology spheres and formulated in terms of splice diagrams. In terms of graph decompositions it is proved by Chaves (\cite[Th~2.2.10]{NC96}). For a short survey and proof, see \cite[Th.~2.11]{PichSea:barfg}.

\begin{thm}[{\cite[Th.~2.11]{PichSea:barfg}}]\label{fibration}
Let $(M,L)$ be a plumbing link with plumbing tree $\Gamma$ and intersection matrix $M_{\Gamma}$. Let $L = \epsilon_1 K_1 \cup \ldots \cup \epsilon_n K_n$ and let $v_1, \ldots, v_s$ be the vertices of \ $\Gamma$. Let 
\begin{equation*}
b(L)=(b_1,\ldots,b_s) \in \Z{}^{s} \ ,
\end{equation*}
where $b_i$ is the sum of the multiplicities $\epsilon_j$ carried by the arrows attached to the vertex $v_i$. Then $L$ is fibred if there exist $(m_1,\ldots,m_s) \in \Z{}^{s}$ such that the next two conditions hold:
\begin{enumerate}[i)]
\item The following system of equations is satisfied:\label{monsys}
\begin{equation}\label{monsyst}
M_{\Gamma} \leftexp{t}{(m_1,\ldots,m_s)} + \leftexp{t}{b(L)} = 0 \ ,
\end{equation}
where $\leftexp{t}(\cdot)$ means transposition,
\item for each node $v_j$ of \ $\Gamma$, the integer $m_j \neq 0$.\label{multdiffzero}
\end{enumerate}
\end{thm}

\begin{defin}
The system of equations \eqref{monsyst} is called the \defi{monodromical system} of $L$ (see \cite[Def~4.2]{MR1824957}).
\end{defin}

Note that the solution to the system of equations \eqref{monsyst} is unique when the matrix $M_{\Gamma}$ is non-degenerate; for example, when it is negative definite.

For each vertex $v_i$ of the plumbing tree $\Gamma$, let $V_i$ be the intersection of $M$ and the $\D^2$-fibre bundle corresponding to $v_i$. Then $V_i$ is a $\Sp{1}$-bundle. 

In the proof of Theorem~\ref{fibration} one obtains that any fibration $\phi \colon M \setminus L \to \Sp{1}$ can be modified by an isotopy in such a way that each fibre of $\phi$ is transverse to all the plumbing tori of $M$ and to all the $\Sp{1}$-fibres of any $V_i$ such that $m_i \neq  0$.

\begin{rmk}\label{qpm}
Let $\mathfrak{F}$ be a fibre of $\phi$, then the monodromy of the fibration $\phi$ admits a quasi-periodic representative $\mathfrak{h} \colon \mathfrak{F} \to \mathfrak{F}$  whose restriction to $\mathfrak{F}_i = \mathfrak{F} \cap V_i$  coincides with the first return map on $\mathfrak{F}_i$ of the fibres of $V_i$ endowed with the orientation $\epsilon_i K$, where $K$ is oriented as the boundary of a $\D^2$-fibre of a plumbed bundle, and where $\epsilon_i = \frac{m_i}{|m_i|}$.
\end{rmk}

In particular, one has: 
 
\begin{prop}\label{order}
The order of $\mathfrak{h}$ on $\mathfrak{F}_i$ equals $|m_i|$.
\end{prop}

\subsection{The Nielsen graph of the monodromy $\mathfrak{h}$.}\label{ngmh}
In this section we present the construction of the Nielsen graph of a quasi-periodic diffeomorphism. First, let us recall the concept of Nielsen graph for a periodic diffeomorphism (see also \cite[p.~347]{MR1824957}). This concept is based in the theory developed by Nielsen in \cite{Nielsen:strk} and \cite{MR0015791}.

Let $\calm{S}$ be an oriented, compact, connected surface and let $\tau \colon \calm{S} \rightarrow \calm{S}$ be an orientation preserving periodic diffeomorphism of order $m \geq 2$. Then $\tau$ generates an action of the group $\Z{m}$ on $\calm{S}$. 

Let $\Os$ be the orbit space of this action, \ie $\Os$ is the quotient of $\calm{S}$ under the following equivalence relation: given $x,y \in \calm{S}, x \sim y$ if and only if exists $k \in \Z{}$ such that $\tau^{k}(x)=y$. Thus $\Os$ is an orbifold of dimension $2$ homeomorphic to an orientable, compact, connected surface.

Let $\varpi \colon \calm{S} \rightarrow \Os$ be the projection onto the orbit space of $\tau$. There exists a finite number of points $p_1, \ldots, p_s$ whose orbits under $\tau$ are of cardinality $n_i < m$ with $i=1, \ldots, s$. Let $p_i \in \calm{S}$ be one of them, the orbit $\varpi^{-1}(p_i)$ is called an \textbf{exceptional orbit} of $\tau$. Then $\varpi$ is a cyclic branched $m$-covering whose ramification locus is the set of these exceptional orbits in $\Os$.

Let $\lambda_i \in \Z{}$ be defined as
\begin{equation*}
\lambda_i = \frac{m}{n_i} \geq 2 \ \text{for} \ i=1, \ldots, s \ .
\end{equation*}

If $\partial \calm{S} = \emptyset$, let $p_i \in \Os$ be a point representing an exceptional orbit $O_i \in \calm{S}$ of cardinality $n_i$ and let $\D_i$ be a small $2$-disc with centre $p_i$, \ie such that each $x \in \D_i \setminus \{p_i\}$ represents an orbit of cardinality $m$. Then $\pr{\varpi}(\D_i)$ consist of $n_i$ disjoint discs $D_{i,1}, \ldots, D_{i,n_i}$, which are cyclically exchanged by $\tau$. Let $D_{i,j}$ be one of them such that the disc $D_{i,j}$ is oriented as $\calm{S}$. Let us endow its boundary $\partial D_{i,j}$ with the induced orientation. Then $\tau|_{D_{i,j}}^{n_i} : D_{i,j} \rightarrow D_{i,j}$ is conjugate to a rotation of angle $2\pi\frac{\omega_i}{\lambda_i}$ with $0 < \omega_i < \lambda_i$ and $\omega_i$ prime relative to $\lambda_i$. The orientation convention for $D_{i,j}$ and its boundary is essential to obtain a well-defined angle. Let $\sigma_i \in \Z{\lambda}$ be such that $\omega_i \sigma_i \equiv 1 \pmod{\lambda_i}$.

\begin{defin}
The pair $(\lambda_i , \sigma_i)$  is the \textbf{valency} of $\tau$ at $p_i$ (or the valency of $\tau$ for the orbit $\pi(p_i)$).
\end{defin}

If $\calm{S}$ has a non-empty boundary, then $\partial \Os \neq \emptyset$. Let $\widehat \Os$ be the closed oriented surface obtained by attaching a $2$-disc $D'_i$ on each boundary component of $\Os$ and let $\widehat{\calm{S}}$ be the surface obtained by attaching a $2$-disc on each boundary component of $\calm{S}$. Let $\widehat \tau$ be the conical extension of $\tau$ to $\widehat{\calm{S}}$. It may be that $\widehat \tau$ is not differentiable at the centre of the new discs but this is unimportant.

Then $\widehat \Os$ is the orbit space of $\widehat \tau$, given by the action of the group $\Z{m}$ on the surface $\widehat{\calm{S}}$.

Let $p_i'$ be the centre of one of the discs $D'_i$. We define the valency for the orbit of a boundary component of $\Os$ as the valency of $\tau$ at $p_i'$ in the same way as for the exceptional orbits.

Notice that the boundary components of $\calm{S}$ are oriented as the boundary of the attached discs and not as the boundary of $\calm{S}$. 

From these valencies, one can construct a graph representing the diffeomorphism $\tau$: 

\begin{defin}
Let $\Ng{\tau}$ be the graph constructed in the following way:
\begin{itemize}
\item The graph $\Ng{\tau}$ has only one vertex, representing the surface $\Os$. This vertex is weighted by the pair $[m,g]$ where $m$ is the order of $\tau$ and $g$ is the genus of $\Os$,
\item we attach to the vertex of $\Ng{\tau}$ one \defi{stalk}\index{stalk} (\tikz[baseline] \filldraw[black] (0pt,3pt) -- (20pt,3pt) circle (1.2pt);) for each exceptional orbit and we weight it by the valency for the corresponding exceptional orbit,
\item we attach to the vertex of $\Ng{\tau}$ one \defi{boundary-stalk}\index{boundary stalk} (\tikz[baseline] \draw (0pt,3pt) -- (18.5pt,3pt) (20pt,3pt) circle (1.2pt);) for each boundary component of $\Os$ and we weight it by the valency for the corresponding boundary component.
\end{itemize}
The graph $\Ng{\tau}$ is called the \textbf{Nielsen graph} of the periodic diffeomorphism $\tau$ and the collection $\left( m,g, (\lambda_1, \sigma_1), \ldots, (\lambda_{s'}, \sigma_{s'})\right)$ consists of the \textbf{Nielsen invariants} of $\tau$.
\end{defin}

Figure~\ref{fig:ngrabs} shows the Nielsen graph $\Ng{\tau}$ of a diffeomorphism $\tau \colon \calm{S} \to \calm{S}$ of order $m$ with $s$ exceptional orbits, where $\Os$ has genus $g$ and $s'-s$ boundary components.

\begin{figure}[H]
\begin{center}
\begin{tikzpicture}[xscale=1,yscale=.8]
\draw (0cm,0cm) -- (-3.5cm,1.5cm);
\draw (0,0) -- (-3.5,-1.5);
\draw (0,0) -- (3.44,1.46);
\draw (0,0cm) -- (3.44,-1.46);
\draw (3.5,1.5) circle (2pt);
\draw (3.5,-1.5) circle (2pt);
\filldraw [black] 
(0,0) circle (4pt)
(-3.5,1.5) circle (2pt)
(-3.5,-1.5) circle (2pt);
\draw (-1.38cm,1.1cm) node{$(\lambda_1, \sigma_1)$} (-1.34cm,-1.1cm) node{$(\lambda_s, \sigma_s)$} (1.1cm,1.1cm) node{$(\lambda_{s+1}, \sigma_{s+1})$} (1.34cm,-1.1cm) node{$(\lambda_{s'}, \sigma_{s'})$} (0cm,-.4cm) node {$[m,g]$};
\draw[thick,loosely dotted] (-2cm,.5cm) arc (165:195:2cm) (2cm,-.5cm) arc (-15:15:2cm);
\end{tikzpicture}
\end{center}
\caption{A Nielsen graph with $s$ stalks and $s'-s$ boundary stalks.}
\label{fig:ngrabs}
\end{figure}
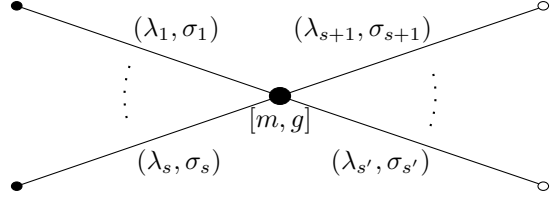

Let $(M,L)$ be a fibred plumbing link and let $\mathfrak{h} \colon \mathfrak{F} \to \mathfrak{F}$ be the quasi-periodic representative of the monodromy of the fibration $\phi$, mentioned in Remark~\ref{qpm}.

\begin{defin}[{\cite[\S~3]{MR1824957}}]
Given the quasi-periodic diffeomorphism $\mathfrak{h} \colon \mathfrak{F} \to \mathfrak{F}$, let $\cc$ be a family  of disjoint simple closed curves in $\mathfrak{F}$ such that
for each connected component $c \in \cc$ one can choose a small annulus $\mc{U}(c) \subset \mathfrak{F}$, tubular neighbourhood of $c$ with the following properties:
\begin{itemize}
\item[-] for any pair of distinct curves $c_i, c_j \in \cc$, we have that $\mc{U}(c_i) \cap \mc{U}(c_j) = \emptyset$,
\item[-] $\mc{U}(\cc) =\mathfrak{h}(\mc{U}(\cc))$ and $\mathfrak{h}(\cc)=\cc$,
\item[-] the restriction  of $\mathfrak{h}$ to the complement of
\begin{equation*}
\mathring{\mc{U}}(\cc) = \bigcup_{c \in \cc} \mathring{\mc{U}}(c)
\end{equation*}
is periodic, where $\mathring{\mc{U}}(c)$ is the interior of $\mc{U}(c)$.
\end{itemize}
The family $\cc$ is called a \textbf{reduction system} of curves for the diffeomorphism $\mathfrak{h}$.
\end{defin}

By \cite[page 348]{MR1824957} we can construct the Nielsen graph of the quasi-periodic diffeomorphism $\mathfrak{h}$ in the following way.

Let $\cc$ be a minimal reduction system for $\mathfrak{h}$. Let $\mc{G}_{\mathfrak{h}}$ be the following graph:
\begin{itemize}
\item $\mc{G}_{\mathfrak{h}}$ has one vertex for each connected component of $\mathfrak{F} \setminus \cc$,
\item let $\mathfrak{F}_i$ and $\mathfrak{F}_j$ be connected components of $\mathfrak{F} \setminus \cc$ such that there is a curve $c \in \cc$ with $c \subset \overline{\mathfrak{F}_i} \cap \overline{\mathfrak{F}_j}$, where $\overline{\mathfrak{F}_i}$ and $\overline{\mathfrak{F}_j}$ are the closures of $\mathfrak{F}_i$ and $\mathfrak{F}_j$ respectively. Then $\mc{G}_{\mathfrak{h}}$ has an edge between the vertices $u_i$ and $u_j$ corresponding to each such curve $c$.
\end{itemize}

Now, let $\mc{U}(\cc)$ be a small tubular neighbourhood of $\cc$ and let $\overline{\mc{G}_{\mathfrak{h}}}$ be the quotient graph of the induced action of $\mathfrak{h}$ on the graph $\mc{G}_{\mathfrak{h}}$, where one vertex $i$ of $\overline{\mc{G}_{\mathfrak{h}}}$ represents one connected component $\mathfrak{F}_i$ of $\mathfrak{F} \setminus \mc{U}(\cc)$ if $\mathfrak{h}(x) \in \mathfrak{F}_i$ for all $x \in \mathfrak{F}_i$, and a vertex $i$ represents $q_{i}$ (with $q_i > 1$) connected components of $\mathfrak{F} \setminus \mc{U}(\cc)$ if $\mathfrak{h}$ permutes cyclically these $q_{i}$ components.

Let $i$ be a vertex of $\overline{\mc{G}_{\mathfrak{h}}}$ such that $i$ represents $q_{i}$ connected components of $\mathfrak{F} \setminus \mc{U}(\cc)$, $\mathfrak{F}_{i,j}$ where $1 \leq j \leq q_i$. Let $\mathfrak{F}_{i,j}$ be one of them.

Let $\mathfrak{h}_{i}$ be the diffeomorphism defined by $\mathfrak{h}_{i}=\mathfrak{h}^{q_{i}} \colon \mathfrak{F}_{i,j} \to \mathfrak{F}_{i,j}$, then $\mathfrak{h}_{i}$ is a periodic diffeomorphism with order $m_{i}$. Then $\mathfrak{h}_i$ generates an action of the group $\Z{m_i}$ on $\mathfrak{F}_{i,j}$. Let $\Os_i$ be the orbit space of this action and let $g_{i}$ be the genus of the orbit space $\Os_{i}$ of $\mathfrak{F}_{i,j}$ by $\mathfrak{h}_{i}$.

For each vertex $i$ of $\overline{\mc{G}_\mathfrak{h}}$ we construct the Nielsen graph $\Ng{\mathfrak{h}_{i}}$ and we complete the numerical information by weighting each vertex with the number $q_{i}$ (see Figure~\ref{fig:ngrssep}).

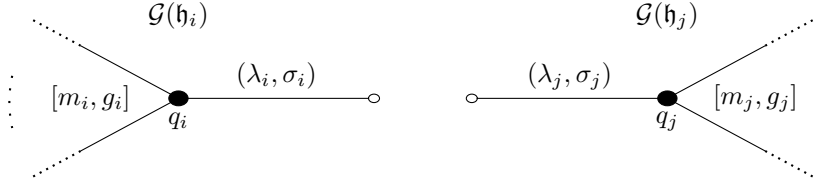
\begin{figure}[H]
\begin{center}
\begin{tikzpicture}[xscale=1.3,yscale=1]
\filldraw[black] (0cm,0cm) circle (2.8pt);
\draw (-1cm,.7cm) -- (0cm,0cm) -- (1.95cm,0cm) (2cm,0cm) circle (1.6pt) (-1cm,-.7cm) -- (0cm,0cm);
\draw[thick,dotted] (-1cm,.7cm) -- (-1.5cm,1.05cm) (-1cm,-.7cm) -- (-1.5cm,-1.05cm);
\filldraw[black] (5cm,0cm) circle (2.8pt);
\draw (6cm,.7cm) -- (5cm,0cm) -- (3.05cm,0cm) (3cm,0cm) circle (1.6pt) (6cm,-.7cm) -- (5cm,0cm);
\draw[thick,dotted] (6cm,.7cm) -- (6.5cm,1.05cm) (6cm,-.7cm) -- (6.5cm,-1.05cm);
\draw (1cm,.3cm) node {$(\lambda_i,\sigma_i)$}
(0cm,-.3cm) node {$q_{i}$}
(-.9cm,0cm) node {$[m_{i},g_{i}]$}
(4cm,.3cm) node {$(\lambda_j,\sigma_j)$}
(5cm,-.3cm) node {$q_{j}$}
(5.9cm,0cm) node {$[m_{j},g_{j}]$}
(0cm,1.1cm) node {$\Ng{\mathfrak{h}_{i}}$}
(5cm,1.1cm) node {$\Ng{\mathfrak{h}_{j}}$};
\draw[thick,loosely dotted] (-1.7cm,.3cm) arc (170:190:2cm) (6.7cm,-.3cm) arc (-10:10:2cm);
\end{tikzpicture}
\end{center}
\caption{Nielsen graphs $\Ng{\mathfrak{h}_{i}}$ and $\Ng{\mathfrak{h}_{j}}$.}
\label{fig:ngrssep}
\end{figure}

\begin{defin}[{\cite[page~347]{MR1824957}}]
Given the quasi-periodic diffeomorphism $\mathfrak{h} \colon \mathfrak{F} \to \mathfrak{F}$, let $\cc$ be a reduction system for $\mathfrak{h}$. Let $c \in \cc$ be a simple closed curve in $\mathfrak{F}$ and let $\mc{U}(c) \subset \mathfrak{F}$ be as before, then there exists an orientation preserving diffeomorphism $\mu \colon [-1,1] \times \Sp{1} \to \mc{U}(c)$ such that $\mu(\{0\} \times \Sp{1})=c$.

Let $N$ be the smallest integer such that
\begin{equation*}
\mathfrak{h}^{N}|_{\mathfrak{F} \setminus\mc{U}(\cc)} = \id_{\mathfrak{F} \setminus\mc{U}(\cc)} \ ,
\end{equation*}
then the restriction $\mathfrak{h}^{N}$ to $\mc{U}(c)$ is a Dehn twist and the restriction of $\mathfrak{h}$ is characterised by a rational number $t$ in the following way: Consider the path $\gamma$ in $\mc{U}(c)$ defined by $\gamma(s)= \mu(s,e^{i \theta})$ where $\theta$ is fixed and $s \in [-1,1]$. We orient $\gamma$ by $[-1,1]$ and then, we orient $c$ in such a way that $\gamma \cdot c = +1$ in $\Ho{1}{\mc{U}(c),\Z{}}$. Then there exists $K \in \Z{}$ such that the cycles $Kc$ and $\mathfrak{h}^{N}(\gamma)-\gamma$ are homologous in $\mc{U}(c)$.

The rational number $t= \frac{K}{N}$ is called the \textbf{twist number} of $\mathfrak{h}$ along $c$.
\end{defin}

Now, let $A$ be an edge of the graph $\overline{\mc{G}_{\mathfrak{h}}}$ connecting the vertices $i$ and $j$, \ie it represents a curve $c \in \cc$ such that $c \subset \overline{\mathfrak{F}_{i}} \cap \overline{\mathfrak{F}_{j}}$. Let $t$ be the twist of $\mathfrak{h}$ along $c$. The two boundary components of $\mc{U}(c)$ are represented by two boundary-stalks in the Nielsen graphs $\Ng{\mathfrak{h}_{i}}$ and $\Ng{\mathfrak{h}_{j}}$ respectively.

Then we take the disjoint union of the Nielsen graphs $\Ng{\mathfrak{h}_{i}}$ and $\Ng{\mathfrak{h}_{j}}$, and we replace the two boundary-stalks by a single edge joining the vertices of the Nielsen graphs. We weight this edge with the twist $t$ and the valencies of the eliminated boundary-stalks at its extremes (see Figure~\ref{fig:jngrs}).

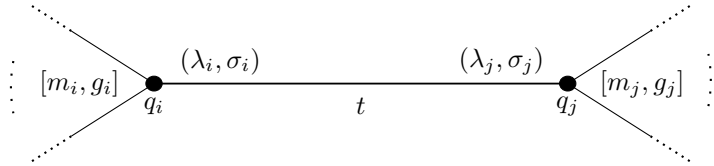
\begin{figure}[H]
\begin{center}
\begin{tikzpicture}[xscale=1.1,yscale=1]
\filldraw[black] (0cm,0cm) circle (2.8pt);
\draw (-1cm,.7cm) -- (0cm,0cm) (5cm,0cm) -- (6cm,.7cm) (-1cm,-.7cm) -- (0cm,0cm) (6cm,-.7cm) -- (5cm,0cm);
\draw[line width=.69pt] (0cm,0cm) -- (5cm,0cm); 
\draw[thick,dotted] (-1cm,.7cm) -- (-1.5cm,1.05cm) (-1cm,-.7cm) -- (-1.5cm,-1.05cm);
\filldraw[black] (5cm,0cm) circle (2.8pt);
\draw[thick,dotted] (6cm,.7cm) -- (6.5cm,1.05cm) (6cm,-.7cm) -- (6.5cm,-1.05cm);
\draw (.8cm,.3cm) node {$(\lambda_i,\sigma_i)$}
(0cm,-.3cm) node {$q_{i}$}
(-.9cm,0cm) node {$[m_{i},g_{i}]$}
(4.2cm,.3cm) node {$(\lambda_j,\sigma_j)$}
(5cm,-.3cm) node {$q_{j}$}
(5.9cm,0cm) node {$[m_{j},g_{j}]$}
(2.5cm,-.3cm) node {$t$};
\draw[thick,loosely dotted] (-1.7cm,.3cm) arc (170:190:2cm) (6.7cm,-.3cm) arc (-10:10:2cm);
\end{tikzpicture}
\end{center}
\caption{Joining the Nielsen graphs $\Ng{\mathfrak{h}_{i}}$ and $\Ng{\mathfrak{h}_{j}}$.}
\label{fig:jngrs}
\end{figure}

We repeat this process for all the edges in the graph $\overline{\mc{G}_{\mathfrak{h}}}$ and we obtain a new graph $\Ng{\mathfrak{h}}$.

The graph $\Ng{\mathfrak{h}}$ is the Nielsen graph of $\mathfrak{h}$.

\subsection{The Nielsen graph of the diffeomorphism $\mathfrak{h}^r$.}\label{ngmhr}

In order to describe the topology of the link $L_F$ as a graph manifold, we are interested to study the diffeomorphism $\mathfrak{h}^r$, given the quasi-periodic diffeomorphism $\mathfrak{h}$. The following results allow us to construct the Nielsen graph of the diffeomorphism $\mathfrak{h}^r$ from the Nielsen graph $\Ng{\mathfrak{h}}$. They are Lemma~2.2 and Lemma~2.3 of \cite{MR1709489} and their proofs appear there.

\begin{lem}\label{nihr}
Let $\mathfrak{h} \colon \mathfrak{F} \to \mathfrak{F}$ be a periodic diffeomorphism preserving the orientation of the surface $\mathfrak{F}$ and let $r \geq 2$ be an integer. Let $\pi \colon \mathfrak{F} \to \Os$ be the projection onto the orbit space $\Os$ of $\mathfrak{h}$. Let $\left( m,g, (\lambda_1, \sigma_1), \ldots, (\lambda_{s}, \sigma_{s}), (\lambda_{s+1}, \sigma_{s+1}) \ldots, (\lambda_{s'}, \sigma_{s'})\right)$ be the Nielsen invariants of $\mathfrak{h}$. Let $n=\gcd(m,r)$ and let $n_i=\gcd(m/ \lambda_i, r)$.

Let $\ptp{r} \colon \mathfrak{F} \to \ostp{r}$ be the projection onto the orbit space $\ostp{r}$ of $\mathfrak{h}^{r}$ and let
\begin{equation*}
\rho \colon \ostp{r} \to \Os
\end{equation*}
be the map defined by $\rho \circ \ptp{r} = \pi$. Then $\rho$ is a cyclic branched covering of 
$n$ leaves, whose ramification locus is included in the set of exceptional orbits of $\Os$. Moreover each exceptional orbit of $\Os$ with valency $(\lambda_i, \sigma_i)$ is a possible branching point of $\rho$ with order $\ensuremath{\frac{n}{n_i}}$.

The Nielsen invariants of $\mtp{\mathfrak{h}}{r}$ can be computed from the Nielsen graph $\Ng{\mathfrak{h}}$ as follows:
\begin{itemize}
\item[-] The order of $\mtp{\mathfrak{h}}{r}$ is $\ensuremath{\otp{r}}= \frac{m}{n}$.
\item[-] The genus of $\ostp{r}$ is
\vspace{-4pt}
\begin{equation*}
\gtp{r} = n(g-1) + 1 +\frac{1}{2} \sum_{i=1}^{s'} (n -n_i) \ .
\end{equation*}
\vspace{-2pt}
\item[-] The orbit space $\ostp{r}$ has a maximum of $\displaystyle \sum_{i=1}^{s} n_i$ 
exceptional orbits and $\displaystyle \sum_{i=s+1}^{s'} n_i$ 
boundary curves.
\item[-] To the $i-th$ exceptional orbit of $\Os$ correspond $n_i$ 
orbits of $\ostp{r}$ with $1 \leq  i \leq s$. To the $i-th$ boundary curve of $\Os$ correspond $n_i$ 
boundary curves of $\ostp{r}$ with $s+1 \leq  i \leq s'$. In any case, the valency $(\ltp{r}_i, \stp{r}_i)$ is given by:
\begin{equation}
\ltp{r}_i = \frac{m}{\lambda_i n_i} \quad \text{and} \quad \stp{r} \times \frac{r}{n} \equiv \sigma_i \pmod{\ltp{r}} \ .
\end{equation}
\end{itemize}
When $\ltp{r} = 1$ the corresponding orbit is regular, in any other case it is an exceptional orbit of $\ostp{r}$ (see Figure \ref{fig:ngrabs2}).
\end{lem}

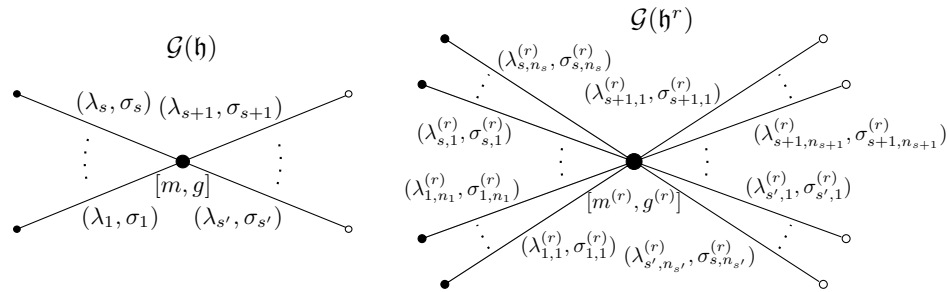
\begin{figure}[H]
\begin{center}
\begin{tikzpicture}
\begin{scope}[xscale=.63,yscale=.6]
\draw (0cm,0cm) -- (-3.5cm,1.5cm);
\draw (0,0) -- (-3.5,-1.5);
\draw (0,0) -- (3.44,1.46);
\draw (0,0cm) -- (3.44,-1.46);
\draw (3.5,1.5) circle (2pt);
\draw (3.5,-1.5) circle (2pt);
\filldraw [black] 
(0,0) circle (4pt)
(-3.5,1.5) circle (2pt)
(-3.5,-1.5) circle (2pt);
\draw (-1.45cm,1.3cm) node[scale=.9]{$(\lambda_s, \sigma_s)$} 
      (-1.32cm,-1.25cm) node[scale=.9]{$(\lambda_1, \sigma_1)$} 
      (.8cm,1.18cm) node[scale=.9]{$(\lambda_{s+1}, \sigma_{s+1})$} 
      (1.17cm,-1.25cm) node[scale=.9]{$(\lambda_{s'}, \sigma_{s'})$} 
      (0cm,-.55cm) node[scale=.9]{$[m,g]$};
\draw[thick,loosely dotted] (-2cm,.5cm) arc (165:195:2cm) (2cm,-.5cm) arc (-15:15:2cm);
\draw (0.2cm,2.5cm) node{$\Ng{\mathfrak{h}}$};
\end{scope}
\begin{scope}[xshift=6cm,scale=.75]
\filldraw (0,0) circle (4pt)
          (0cm,0cm) -- node[scale=.8,below right=0pt and -10pt]{$(\ltp{r}_{1,1}, \stp{r}_{1,1})$} (213:4cm) circle (2pt)
          (0cm,0cm) -- node[scale=.8,above left=-4pt and 0pt]{$(\ltp{r}_{1,n_1}, \stp{r}_{1,n_1})$} (200:4cm) circle (2pt)
          (0cm,0cm) -- node[scale=.8,below left=-5pt and 3pt]{$(\ltp{r}_{s,1}, \stp{r}_{s,1})$} (160:4cm) circle (2pt)
          (0cm,0cm) -- node[scale=.8,above right=8pt and -19pt]{$(\ltp{r}_{s,n_s}, \stp{r}_{s,n_s})$} (147:4cm) circle (2pt)
;
\draw[thick, loosely dotted] (211:3cm) arc (211:202:3cm)
                             (158:3cm) arc (158:149:3cm)
                             (191:1.3cm) arc (191:165:1.3cm);
\draw (0cm,0cm) -- node[scale=.8,above left=-5pt and -1pt]{$(\ltp{r}_{s+1,1}, \stp{r}_{s+1,1})$} (33:3.91cm)
      (33:4cm) circle (2pt)
      (0cm,0cm) -- node[scale=.8,below right=-5pt and 3pt]{$(\ltp{r}_{s+1,n_{s+1}}, \stp{r}_{s+1,n_{s+1}})$} (20:3.91cm)
      (20:4cm) circle (2pt)
      (0cm,0cm) -- node[scale=.8,above right=-4pt and 0pt]{$(\ltp{r}_{s',1}, \stp{r}_{s',1})$} (-20:3.91cm)
      (-20:4cm) circle (2pt)
      (0cm,0cm) -- node[scale=.8,below left=4pt and -14pt]{$(\ltp{r}_{s',n_{s'}}, \stp{r}_{s,n_{s'}})$} (-33:3.91cm)
      (-33:4cm) circle (2pt)
;
\draw (0cm,-.7cm) node[scale=.8] {$[\otp{r},\gtp{r}]$};
\draw[thick, loosely dotted] (-31:3cm) arc (-31:-22:3cm)
                             (22:3cm) arc (22:31:3cm)
                             (-11:1.3cm) arc (-11:15:1.3cm);
\draw (0.5cm,2.5cm) node{$\Ng{\mathfrak{h}^r}$};
\end{scope}
\end{tikzpicture}
\end{center}
\caption{The Nielsen graphs (in one vertex) $\Ng{\mathfrak{h}}$ and $\Ng{\mathfrak{h}^r}$.}
\label{fig:ngrabs2}
\end{figure}

\begin{lem}\label{twhr}
Let $\mathfrak{h} \colon \mathfrak{F} \to \mathfrak{F}$ be a quasi-periodic diffeomorphism and let $\cc$ be a minimal reduction system for $\mathfrak{h}$. Then the family $\cc$ is a reduction system of $\mtp{\mathfrak{h}}{r}$ and if $t$ and $\ttp{r}$ are the twists of $\mathfrak{h}$ and $\mtp{\mathfrak{h}}{r}$ respectively near to a curve $c \in \cc$, then $\ttp{r}=rt$.
\end{lem}

This last result completes the information of the graph $\Ng{\mathfrak{h}^r}$.

\subsection{Open book of the diffeomorphism $\mathfrak{h}^r$}\label{obmh}
Given a quasi-periodic diffeomorphism of a surface, one can construct the associated mapping torus and obtain an open book using the results given in this section. This open book with the corresponding binding is a fibred plumbing link.

Let us first show how to construct a graph describing the Waldhausen decomposition of a plumbing link. 

\begin{defin}
Let $M$ be a $3$-manifold. A \textbf{Waldhausen decomposition} of $M$ is a decomposition of $M$ as a union of a finite number of $3$-manifolds $M_i$, $M = \bigcup M_i$ such that
\begin{enumerate}
\item each $M_i$ is a Seifert manifold,
\item if $i \neq j$, the intersection $M_i \cap M_j$ is either empty or it is the union of the common boundary components, \ie a union of tori.
\end{enumerate}
\end{defin}

Let $(M,L)$ be a plumbing link. A way to represent $(M,L)$ according to the Waldhausen decomposition of $M$ is the following: Let $\W(M,L)$ be a graph constructed in the following way:

\begin{itemize}
\item The graph $\W(M,L)$ has a vertex $i$ for each Seifert component $M_i$ in the Waldhausen decomposition of $M$. 

\item For each exceptional fibre in $M_i \setminus L$ we attach to $i$ a stalk weighted by the corresponding normalised Seifert invariant $(\alpha, \beta)$, \ie $1 \leq \beta < \alpha$. 

\item For each Seifert fibre in $M_i \cap L$ we attach an arrow weighted by the corresponding normalised Seifert invariant $(\alpha, \beta)$, \ie $0 \leq \beta < \alpha$.

\item  The vertex is weighted by the genus $g_i$ of the orbit space of $M_i$ and the Euler obstruction $e(M_i)$ to be defined below.

\item Let $M_i$ and $M_j$ be two Seifert components in the Waldhausen decomposition of $M$ and let $i$ and $j$ be the corresponding vertices; there is an edge between $i$ and $j$ if and only if the intersection $M_i \cap M_j$ is not empty.

\item Each edge is oriented by the triplet $(\varepsilon, \alpha,\beta)$ (defined as in \cite[\S~1]{Neu:calcplumb}) in the following way: Let $T$ be a separating torus between two Seifert components $M_i$ and $M_j$ (represented by vertices $i$ and $j$ respectively). Let $\mc{U}(T)$ be a thickened torus, small neighbourhood of $T$ and let $T_i \subset M_i$ and $T_j \subset M_j$ be its boundary components. Let us orient $T_i$ and $T_j$ as the boundary of $\mc{U}(T)$. Let $b_i \subset T_i$ be a Seifert fibre of $M_i$ and let $a_i \subset T_i$ be a curve such that $a_i \cdot b_i = 1$ in $\Ho{1}{T_i, \Z{}}$. In the same way, we choose $a_j, b_j \subset T_j$ such that $a_j \cdot b_j = 1$ in $\Ho{1}{T_j, \Z{}}$. Let $g \colon T_i \to T_j$ be an orientation reversing diffeomorphism, induced by the product structure of \ $\overline{\mc{U}(T)}$. There exists some integers $\varepsilon \in \{-1, 1\}$, $\alpha > 0$ and $\beta, \beta' \in \Z{}$ such that
\begin{equation*}
\varepsilon \pr{g}(b_j)=\alpha a_i + \beta b_i  \ \ \text{in} \ \ \Ho{1}{T_i, \Z{}}
\end{equation*}
and
\begin{equation*}
\varepsilon g(b_i)=\alpha a_j + \beta' b_j  \ \  \text{in} \ \ \Ho{1}{T_j, \Z{}}.
\end{equation*}
Moreover, it is possible to choose the curves $a$ and $a'$ in such a way that the integers $\beta$ and $\beta'$ are normalised, \ie $0 \leq \beta < \alpha$ and $0 \leq \beta' < \alpha$. If $\alpha > 1$, $\beta$ and $\beta'$ satisfy the following relation:
\begin{equation*}
\beta \beta' \equiv 1 \pmod{\alpha} \ .
\end{equation*}
If $\alpha = 1$, then $\beta = \beta' = 0$.

Thus, if the edge is oriented from $i$ to $j$, the corresponding triplet is $(\varepsilon, \alpha,\beta)$; if the orientation is from $j$ to $i$, we write the triplet $(\varepsilon, \alpha,\beta')$.

\item Let $M_i$ be a Seifert component; for each boundary component of $M_i$ we have a normalised pair $(\alpha, \beta)$ which is related to the choice of a section on the boundary of $M_i$; the Euler obstruction $e(M_i)$ is the obstruction to extend that section to all $M_i$.
\end{itemize}

In Figure \ref{fig:wggl} it is shown the graph $\W(M,L)$ of a plumbing link $(M,L)$ with $s$ exceptional fibres in $M_i \setminus L$, $s'-s$ Seifert fibres in $M_i \cap L$ and Euler obstruction $e_i=e(M_i)$.

\begin{figure}[H]
\begin{center}
\begin{tikzpicture}[xscale=1.3,yscale=1]
\filldraw[black] (0cm,-5.5cm) circle (2.8pt) (-.6cm,-7.3cm) circle (1.6pt) (-1.7cm,-6.6cm) circle (1.6pt);
\draw (-.6cm,-3.7cm) -- (0cm,-5.5cm) -- (-.6cm,-7.3cm) 
(-1.7cm,-4.4cm) -- (0cm,-5.5cm) -- (-1.7cm,-6.6cm) 
(5cm,-4.8cm) --  (4cm,-5.5cm) -- (5cm,-6.2cm)
(-1.61cm,-4.65cm) -- (-1.7cm,-4.4cm) -- (-1.45cm,-4.4cm) (-.65cm,-3.9cm) -- (-.6cm,-3.7cm) -- (-.4cm,-3.8cm) (1.8cm,-5.4cm) -- (2cm,-5.5cm) -- (1.8cm,-5.6cm);
\draw[line width=.69pt] (0cm,-5.5cm) -- (4cm,-5.5cm); 
\filldraw[black] (4cm,-5.5cm) circle (2.8pt);
\draw[thick,dotted] (5cm,-4.8cm) -- (5.5cm,-4.45cm) (5cm,-6.2cm) -- (5.5cm,-6.55cm);
\draw (2cm,-5cm) node {$(\varepsilon,\alpha,\beta)$} (.17cm,-4.4cm) node [scale=.95] {$(\alpha_{s'},\beta_{s'})$} (.4cm,-5.8cm)  node {$[e_i,g_i]$} (-1.4cm,-5.3cm) node [scale=.95] {$(\alpha_{s+1}, \beta_{s+1})$} (-1.4cm,-5.9cm) node [scale=.95] {$(\alpha_{s}, \beta_{s})$} (.15cm,-6.6cm) node [scale=.95] {$(\alpha_1,\beta_1)$} (3.6cm,-5.8cm) node {$[e_j,g_j]$};
\draw[thick,loosely dotted] (-.8cm,-4.3cm) arc (124:134:1.5cm) (-.8cm,-6.7cm) arc (246:236:1.5cm) (5.7cm,-5.8cm) arc (-10:10:2cm);
\end{tikzpicture}
\end{center}
\caption{Graph $\W(M,L)$ of the plumbing link $(M,L)$.}
\label{fig:wggl}
\end{figure}
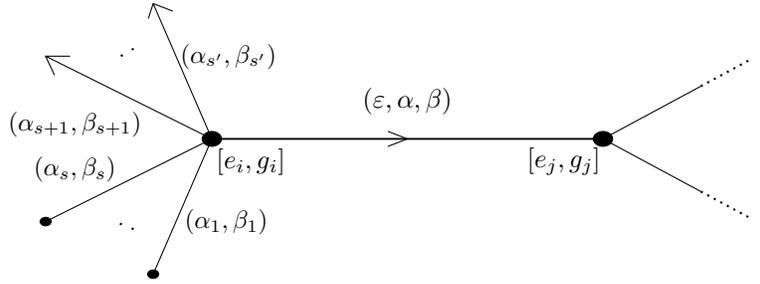

Now, given a periodic diffeomorphism $\tau \colon \calm{S} \to \calm{S}$ of a surface $\calm{S}$, we obtain the mapping torus $T(\tau)$ as a Seifert manifold. Recall that the mapping torus $T(\tau)$ of $\tau$ is the quotient of the product $\calm{S} \times [0,1]$ by the equivalence relation $(x, 1) \sim (\tau(x), 0)$ for all $x \in \calm{S}$.

The following result is proved in \cite[Section~4.4]{MR915761}.

\begin{lem}\label{Mont}
Let $\calm{S}$ be a surface without boundary and let $\tau \colon \calm{S} \to \calm{S}$ be a periodic diffeomorphism with $s$ exceptional orbits. Let $(\lambda_i,\sigma_i)$ be their valencies with $i=1,\ldots,s$ .

Then the mapping torus $T(\tau)$ is a Seifert manifold whose orbit space is the orbit space of $\tau$ and whose Seifert invariants are given as follows :
\begin{itemize}
\item  There are $s$ exceptional fibres whose Seifert invariants are $(\alpha_i,\beta_i) = (\lambda_i,\sigma_i)$ with  $i=1\ldots,s$,  
\item the  rational Euler number is $0$, so the integral Euler obstruction $e$ is given by : 
\begin{equation*}
e = \sum_{i=1}^s \frac{\sigma_i}{\lambda_i} \ .
\end{equation*}
\end{itemize}
\end{lem}

Returning to the case of the quasi-periodic diffeomorphism $\mathfrak{h} \colon \mathfrak{F} \to \mathfrak{F}$, we have that the mapping torus $T(\mathfrak{h})$ has boundary $\partial \mathfrak{F} \times \Sp{1}$. Let $\lambda \in \Sp{1}$, then $\partial \mathfrak{F} \times \{\lambda\}$ is the boundary of a manifold $\mathfrak{F}_{\lambda} \subset T(\mathfrak{h})$ diffeomorphic to $\mathfrak{F}$.
Let
\begin{equation*}
\Op{\mathfrak{h}}= T(\mathfrak{h}) \bigcup_{\partial \mathfrak{F} \times \Sp{1}} (\partial \mathfrak{F} \times \D^{2}) \ ,
\end{equation*}
let $L(\mathfrak{h})= (\partial \mathfrak{F} \times \{0\}) \subset (\partial \mathfrak{F} \times \D^{2})$ and let
\begin{equation*}
\pi_{\mathfrak{h}} \colon (\Op{\mathfrak{h}} \setminus L(\mathfrak{h})) \to \Sp{1}
\end{equation*}
be such that $\pi_{\mathfrak{h}} \left(\left(\partial \mathfrak{F} \times (0,\lambda]\right) \bigcup_{\{\partial \mathfrak{F} \times \lambda\}} \mathfrak{F}_{\lambda}\right)= \lambda$.
 
Then, one can see $\pi_{\mathfrak{h}}$ as an open book fibration of $\Op{\mathfrak{h}}$.

The following result gives information on this open book and it is an adaptation of \cite[Lemma 4.4]{MR1824957}.

\begin{thm}\label{mero}
Let $\mathfrak{h} \colon \mathfrak{F} \to \mathfrak{F}$ be a quasi-periodic diffeomorphism with Nielsen graph $\Ng{\mathfrak{h}}$. Then the pair $(\Op{\mathfrak{h}},L(\mathfrak{h}))$ is a plumbing link (moreover, it is a fibred link) whose corresponding graph $\W(\Op{\mathfrak{h}},L(\mathfrak{h}))$ is obtained as follows. There exists a graph isomorphism from $\Ng{\mathfrak{h}}$ to $\W(\Op{\mathfrak{h}},L(\mathfrak{h}))$ sending:
\begin{itemize}
\item[-] the vertices of $\Ng{\mathfrak{h}}$ to the vertices of $\W(\Op{\mathfrak{h}},L(\mathfrak{h}))$,
\item[-] the edges of $\Ng{\mathfrak{h}}$ to the edges of $\W(\Op{\mathfrak{h}},L(\mathfrak{h}))$,
\item[-] the stalks of $\Ng{\mathfrak{h}}$ to the stalks of $\W(\Op{\mathfrak{h}},L(\mathfrak{h}))$,
\item[-] the boundary-stalks of $\Ng{\mathfrak{h}}$ to the arrows of $\W(\Op{\mathfrak{h}},L(\mathfrak{h}))$.
\end{itemize}
Moreover, 
\begin{itemize}
\item consider a vertex of $\Ng{\mathfrak{h}}$ with genus $g$, order $m$ and with neighbour valencies $(\lambda_i,\sigma_i), i=1,\ldots,s''$ (taking into account all the incident edges, including those corresponding to stalks and boundary-stalks). Then the corresponding  vertex of $\W(\Op{\mathfrak{h}},L(\mathfrak{h}))$ is weighted by $[g,e]$ where the integral Euler obstruction $e$ is given by : 
\begin{equation}\label{f1}
e = \sum_{i=1}^{s''} \frac{\sigma_i}{\lambda_i} \ ,
\end{equation}
\item  for each stalk of $\Ng{\mathfrak{h}}$ with valency $(\lambda_i,\sigma_i)$ (with $1 \leq i \leq s$), the corresponding exceptional fibre has Seifert invariant
\begin{equation}\label{f2}
(\alpha_i,\beta_i) = (\lambda_i,\sigma_i) \ ,
\end{equation}
\item for each boundary-stalk of $\Ng{\mathfrak{h}}$ with valency $(\lambda_i,\sigma_i)$, twist $t_i$ and order $m_i$ in the adjacent vertex (with $s+1 \leq i \leq s'$), the corresponding Seifert invariant is 
\begin{equation}\label{NiWa}
(\alpha_i,\beta_i) = \left(|t_i \lambda_i| \, , \, - \frac{t_i}{|t_i|} \cdot \frac{1-m_it_i\sigma_i}{m_i}\right) \ ,
\end{equation}
\item for each edge of $\Ng{\mathfrak{h}}$, the triplet of the corresponding edge oriented  from left to right is: 
\begin{equation}\label{f3}
(\varepsilon_k,\alpha_k,\beta_k) = \left(- \frac{t_k}{|t_k|} \, , \, |m'_kt_k \lambda_k | \, , \, - \frac{t_k}{|t_k|} \cdot \frac{m'_k-m_i m'_k t_k \sigma_k}{m_i}\right) \ .
\end{equation}
Notice that for each of the three previous equalities, there exists a choice of $\sigma$ in its class modulo $\lambda$ such that the corresponding pair $(\alpha,\beta)$ is normalised, \ie $0 \leq \beta < \alpha$. Let us fix such integers $\sigma$. 
\end{itemize} 
\end{thm}

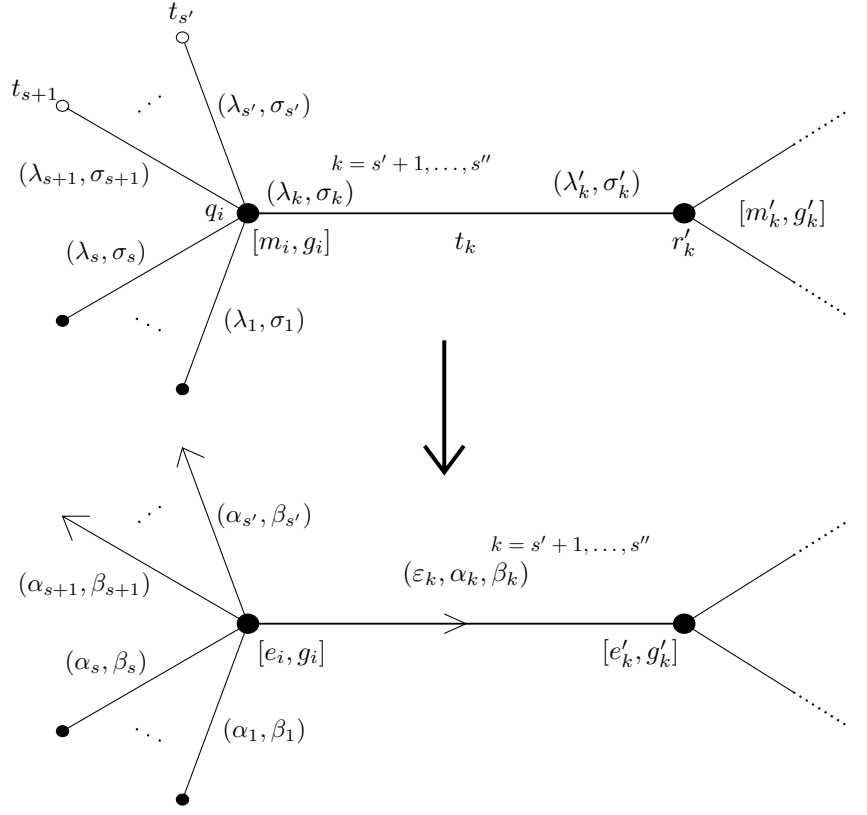
\begin{figure}[H]
\begin{center}
\begin{tikzpicture}[xscale=1.45,yscale=1.3]
\begin{scope}
\filldraw[black] (0cm,0cm) circle (2.8pt) (-.6cm,-1.8cm) circle (1.6pt) (-1.7cm,-1.1cm) circle (1.6pt);
\draw (-.58cm,1.74cm) -- (0cm,0cm) -- (-.6cm,-1.8cm) (-1.66cm,1.076cm) -- (0cm,0cm) -- (-1.7cm,-1.1cm) (5cm,-.7cm) --    (4cm,0cm) -- (5cm,.7cm) (-1.7cm,1.1cm) circle (1.6pt) (-.6cm,1.8cm) circle (1.6pt); 
\draw[line width=.69pt] (0cm,0cm) -- (4cm,0cm); 
\filldraw[black] (4cm,0cm) circle (2.8pt);
\draw[thick,dotted] (5cm,.7cm) -- (5.5cm,1.05cm) (5cm,-.7cm) -- (5.5cm,-1.05cm);
\draw (.57cm,.2cm) node {$(\lambda_k,\sigma_k)$}
      (1.5cm,.5cm) node[scale=.8] {$k=s'+1, \ldots,s''$}
     (.15cm,1.1cm) node [scale=.95] {$(\lambda_{s'},\sigma_{s'})$} (.4cm,-.3cm)  node {$[m_i,g_i]$} (-.3cm,0cm) node {$q_i$} (-1.5cm,.4cm) node [scale=.95] {$(\lambda_{s+1}, \sigma_{s+1})$} (-.6cm, 2.05cm) node {$t_{s'}$} (-1.95cm,1.25cm) node {$t_{s+1}$} 
(-1.3cm,-.4cm) node [scale=.95] {$(\lambda_{s}, \sigma_{s})$} (.15cm,-1.1cm) node [scale=.95] {$(\lambda_1,\sigma_1)$}
(3.2cm,.3cm) node {$(\lambda'_k,\sigma'_k)$}
(4cm,-.3cm) node {$r'_k$}
(4.9cm,0cm) node {$[m'_k,g'_k]$}
(2cm,-.3cm) node {$t_k$};
\draw[thick,loosely dotted] (-.8cm,1.2cm) arc (124:134:1.5cm) (-.8cm,-1.2cm) arc (246:236:1.5cm) (5.7cm,-.3cm) arc (-10:10:2cm);
\end{scope}
\draw[line width=1.5pt,yshift=.5cm,scale=.9] (2cm,-2cm) -- (2cm,-3.5cm) 
     (1.8cm,-3.2cm) -- (2cm,-3.5cm) -- (2.2cm,-3.2cm);
\begin{scope}[yshift=1.3cm]
\filldraw[black] (0cm,-5.5cm) circle (2.8pt) (-.6cm,-7.3cm) circle (1.6pt) (-1.7cm,-6.6cm) circle (1.6pt);
\draw (-.6cm,-3.7cm) -- (0cm,-5.5cm) -- (-.6cm,-7.3cm) 
(-1.7cm,-4.4cm) -- (0cm,-5.5cm) -- (-1.7cm,-6.6cm) 
(5cm,-4.8cm) --  (4cm,-5.5cm) -- (5cm,-6.2cm)
(-1.61cm,-4.65cm) -- (-1.7cm,-4.4cm) -- (-1.45cm,-4.4cm) (-.65cm,-3.9cm) -- (-.6cm,-3.7cm) -- (-.4cm,-3.8cm) (1.8cm,-5.4cm) -- (2cm,-5.5cm) -- (1.8cm,-5.6cm);
\draw[line width=.69pt] (0cm,-5.5cm) -- (4cm,-5.5cm); 
\filldraw[black] (4cm,-5.5cm) circle (2.8pt);
\draw[thick,dotted] (5cm,-4.8cm) -- (5.5cm,-4.45cm) (5cm,-6.2cm) -- (5.5cm,-6.55cm);
\draw (2cm,-5cm) node {$(\varepsilon_k,\alpha_k,\beta_k)$} 
                  (2.96cm,-4.7cm) node[scale=.8] {$k=s'+1, \ldots,s''$}  
(.15cm,-4.4cm) node [scale=.95] {$(\alpha_{s'},\beta_{s'})$} (.4cm,-5.8cm)  node {$[e_i,g_i]$} (-1.5cm,-5.1cm) node [scale=.95] {$(\alpha_{s+1}, \beta_{s+1})$} (-1.3cm,-5.9cm) node [scale=.95] {$(\alpha_{s}, \beta_{s})$} (.15cm,-6.6cm) node [scale=.95] {$(\alpha_1,\beta_1)$} (3.6cm,-5.8cm) node {$[e'_k,g'_k]$};
\draw[thick,loosely dotted] (-.8cm,-4.3cm) arc (124:134:1.5cm) (-.8cm,-6.7cm) arc (246:236:1.5cm) (5.7cm,-5.8cm) arc (-10:10:2cm);
\end{scope}
\end{tikzpicture}
\end{center}
\caption{Isomorphism between the graphs $\Ng{\mathfrak{h}}$ and $\W(\Op{\mathfrak{h}},L(\mathfrak{h}))$.}
\label{fig:isngrwgr}
\end{figure}

Formulae \eqref{f1} and \eqref{f2} are consequences of the Lemma \ref{Mont}. Formula \eqref{f3} is proved in \cite[Lemma 4.4]{MR1824957} as consequence of the following result. Moreover, this lemma, which is an adaptation of a part of \cite[Lemma 4.4]{MR1824957}, enables us to prove also \eqref{NiWa}.

\begin{lem}\label{useful}
Let $A = [-1,1] \times \Sp{1}$ be an annulus with boundary components $c = \{-1\} \times \Sp{1}$ and $c' = \{1\} \times \Sp{1}$ and let $\mathfrak{h} \colon A \to A$ be an orientation preserving diffeomorphism such that 
\begin{itemize}
\item $\mathfrak{h}(c)=c$ and $\mathfrak{h}|_{c}$ is periodic of order $m$,
\item $\mathfrak{h}(c')=D'$ and $\mathfrak{h}|_{c'}$ is periodic of order $m'$.
\end{itemize}
Let $T(\mathfrak{h})$ be the mapping torus of $\mathfrak{h}$ and let $T = c \times \Sp{1}$ and $T' =  c' \times \Sp{1}$ be the boundary components of $T(\mathfrak{h})$. 
Let $b \subset T$ be the curve which is the image in $T(\mathfrak{h})$ of the union of segments
\begin{equation*}
\bigcup_{i=1}^{m} (-1,\mathfrak{h}^{i}(\lambda)) \times [0,1] \ ,
\end{equation*}
where $((-1,\mathfrak{h}^{i}(\lambda)) \times [0,1]) \subset (A \times [0,1])$ and let $b' \subset T'$ be the curve which is the image in $T(\mathfrak{h})$ of the union of segments
\begin{equation*}
\bigcup_{i=1}^{m'} (1,\mathfrak{h}^{i}(\lambda)) \times [0,1] \ .
\end{equation*}
Then the following equation holds
\begin{equation}
mb' = -mm't \lambda a + (m' - mm't \sigma)b \quad \text{in} \quad \Ho{1}{T,\Z{}} \ ,
\end{equation}
where $a$ is a curve on $T$ such that $a \cdot b = 1$ in $\Ho{1}{T, \Z{}}$ and $a'$ is a curve on $T'$ such that $a' \cdot b' = 1$ in $\Ho{1}{T, \Z{}}$.
\end{lem}
\begin{proof}
Let $d = \{0\} \times \Sp{1} \subset A$, we orient $d$ in such a way that $d$ and $c$ are homologous in $\Ho{1}{A, \Z{}}$, where $c$ is oriented as boundary of $A$. Let $\gamma = [0,1] \times \{0\} \subset A$ be oriented as $[0,1]$. Then $\gamma$ is transverse to $d$ and we have
\begin{equation}\label{twist}
\mathfrak{h}^{mm'}(\gamma) - \gamma = mm'td = mm'tc
\end{equation}
in $\Ho{1}{A, \Z{}}$.

On the other hand, the cycle $m'b- mb' + \gamma - \mathfrak{h}^{mm'}(\gamma)$ is the boundary of a $2$-chain in $T(\mathfrak{h})$. Then, from \eqref{twist}, we obtain
\begin{equation}\label{tw2}
mb' = m'b - mm'td \ \ \text{in} \ \ \Ho{1}{T(\mathfrak{h}), \Z{}} \ .
\end{equation}
For $a$ and $b$, we have the following relation:
\begin{equation}\label{tw3}
c = \lambda a + \sigma b \ \ \text{in} \ \ \Ho{1}{T, \Z{}} \ .
\end{equation}
Combining \eqref{twist}, \eqref{tw2} and \eqref{tw3} we obtain
\begin{equation*}
mb' = - mm't \lambda a + (m'-mm't \sigma) b
\end{equation*}
in $\Ho{1}{T(\mathfrak{h}), \Z{}}$.
\end{proof}

\begin{proof}[Proof of equation \eqref{NiWa}]
Let us take a boundary-stalk of $\Ng{\mathfrak{h}}$ with valency $(\lambda,\sigma)$, twist $t$ and order $m$ in the adjacent vertex. As we know, the diffeomorphism $\mathfrak{h}$ is the identity on the boundary of $\mathfrak{F}$. Then, applying Lemma \ref{useful}, we obtain
\begin{equation*}
b = - t \lambda a' + (1-mt \sigma) b' \quad \text{in} \quad \Ho{1}{T(\mathfrak{h}),\Z{}}
\end{equation*}
and on the other hand we have the following equality:
\begin{equation*}
\varepsilon b = \alpha a' + \beta b' \quad \text{in} \quad \Ho{1}{T',\Z{}} \ .
\end{equation*}
From these last two equations we get
\begin{equation*}
\alpha = |t \lambda| \quad \text{and} \quad \beta = - \frac{t}{|t|} \cdot \frac{1-mt\sigma}{m} ,
\end{equation*}
where $\varepsilon = \ensuremath{-\frac{t}{|t|}}$.
\end{proof}

\section{The link $L_F$ as an open book}\label{Prim}
In this section we present a family of real analytic functions $F \colon \C^3 \to \C$ with isolated singularity at the origin and we describe the link $\LF$ as an open book.

Let $f,g \colon (\C^2,0) \to (\C,0)$ be two complex analytic germs such that the real analytic germ $\fbg \colon (\C^2,0) \to (\C,0)$ has an isolated singularity at the origin. Consider the real analytic function $F \colon \C^3 \to \C$  given by
\begin{equation*}
F(x,y,z)=f(x,y)\overline{g(x,y)}+z^r
\end{equation*}
with $r \in \N$.

The description of $L_F$ as an open book is given in terms of the monodromy of the Milnor fibration of $\fbg$.

By \cite[Proposition 1]{MR2922705}, the function $F$ has an isolated singularity at the origin since the function $f\overline{g}$ has an isolated singularity at the origin. Let $L_F= \Sp{5} \cap \pr{F}(0)$ denote the link of the singularity.

For $\epsilon > 0$ sufficiently small, let $\Lfg=\pr{(\fbg)}(0) \cap \Stres_{\epsilon}$ be the link of $\fbg$. We set $\Lfg$ as the oriented link $L_f - L_g$ as in \cite{PichSea:barfg}.

By \cite[Theorem~5.8]{PichSea:barfg}, the function $\fbg$ has Milnor fibration with projection $\mpfg = \frac{\fbg}{\mid \fbg \mid}$; \ie the Milnor fibration $\mpfg$ is an open book fibration of $\Stres_{\epsilon}$ with binding $\Lfg$.

Let $\varepsilon$ be such that $\D^6_{\varepsilon}$ is a Milnor ball for $F$. Let $\varepsilon^{\prime}$ be such that for all $(x,y,z) \in \pr{F}(0)$ with $(x,y) \in \D^4_{\varepsilon^{\prime}}$ we have 
\begin{equation*}
|f(x,y)\overline{g(x,y)}|^{1/r} < \varepsilon \ .
\end{equation*}

Let us consider the polydisc
\begin{equation*}
\BD{6} = \{(x,y,z) | (x,y) \in \D^4_{\varepsilon^{\prime}}, |z| \leq \varepsilon \} \ .
\end{equation*}

By \cite[Proposition~1.7 and Application~3.8]{durfee:neighalgs} $\LF$ is homeomorphic to the intersection $\pr{F}(0) \cap \partial \BD{6}$. In the sequel we will also denote this intersection by $\LF$.

The following proposition is an adaptation of \cite[Proposition~1.5]{MR1709489}, which deals with the case $f(x,y)+z^k$ where $f$ is holomorphic and reduced.

\begin{prop}\label{diagc}
Let $\cP\colon \LF \to \Stres_{\varepsilon'}$ be the projection defined by
\begin{equation*}
\cP(x,y,z)=(x,y)
\end{equation*}
and let $\Lp = \pr{\cP}(\Lfg)$. Define $\rho_r \colon \C \to \C$ by $\rho_r(z)=z^r$, let $\mpp \colon \LF \setminus \Lp \to \Sp{1}$ be the map given by $\mpp=\frac{z}{\mid z \mid}$ and let $h$ be the monodromy of the Milnor fibration $\mpfg$. Then
\begin{enumerate}[1)]
\item  the following diagram commutes:\label{1}
\begin{equation}\label{dgcm}
\xymatrix{
\LF\setminus \Lp \ar[d]_{\mpp} \ar[r]^{\cP} & \Stres \setminus \Lfg \ar[d]^{\mpfg} \\
\Sp{1} \ar[r]^{-\rho_r} & \Sp{1} \\
}
\end{equation}
\item $\cP $ is a cyclic branched $r$-covering with ramification locus $\Lfg$ and the restriction\label{2}
\begin{equation*}
\cP \colon \Lp \to \Lfg
\end{equation*}
is a homeomorphism,
\item the projection $\mpp$ is an open book fibration with binding $\Lp$,\label{3}
\item the fibres of \ $\mpp$ and $\mpfg$ are diffeomorphic and the monodromy $\hp$ of $\mpp$ is equal to $\mtp{h}{r}$ up to conjugacy in the mapping class group of the fibre.\label{4}
\end{enumerate}
\end{prop}

\begin{proof}
Let us prove \ref{1}). Let $(x,y,z) \in \LF \setminus \Lp$. Then $ 0 \neq z^r=-f(x,y)\overline{g(x,y)}$
and 
\begin{equation*}
\mpfg\left(\cP(x,y,z)\right)=\mpfg(x,y)=\frac{f(x,y)\overline{g(x,y)}}{|f(x,y)\overline{g(x,y)}|} \ .
\end{equation*}

On the other hand, $\ensuremath{\mpp(x,y,z)=\frac{z}{|z|}}$ and
\begin{equation*}
-\rho_r\left(\frac{z}{|z|}\right)= - \frac{z^r}{|z|^r}= \frac{f(x,y)\overline{g(x,y)}}{|f(x,y)\overline{g(x,y)}|} \ . 
\end{equation*}
Hence $\mpfg(\cP(x,y,z))=-\rho_r(\mpp(x,y,z))$.

Now, in order to prove \ref{2}), first we prove that the diagram \eqref{dgcm} is a pull-back diagram.

Let $Q$ be the pull-back of $\mpfg$ by $-\rho_r$ defined by
\begin{align*}
Q &= \{(x,y, \lambda) \in (\Stres \setminus \Lfg) \times  \Sp{1} \mid \mpfg(x,y) = -\rho_r(\lambda)\} \\
  &= \{(x,y, \lambda) \in (\Stres \setminus \Lfg) \times  \Sp{1} \mid \frac{f(x,y)\overline{g(x,y)}}{|f(x,y)\overline{g(x,y)}|} = -\lambda^r \}
\end{align*}

Then, by the universal property of the pull-back, we have the following diagram:
\begin{equation*}
\xymatrix{
\LF \setminus \Lp \ar@/_/[dddr]_{\mpp} \ar[dr]^-{p} \ar@/^/[drrr]^{\cP} & & &\\
                & Q \ar[dd]^{\pi_3} \ar[rr]^-{\pi_{1,2}}           & & \Stres \setminus \Lfg \ar[dd]^{\mpfg} \\
                &                                            & & \\
                & \Sp{1} \ar[rr]^{-\rho_r}                   & & \Sp{1} \\
}
\end{equation*}
where $\pi_3$ is the projection on the third coordinate, $\pi_{1,2}$ is the projection on the first two coordinates and $p \colon (\LF \setminus \Lp) \to Q$ is defined by
\begin{equation*}
p(x,y,z) = \left(x,y, \frac{z}{|z|}\right) \ .
\end{equation*}
Let us define $q \colon Q \to (\LF \setminus \Lp)$ by $q(x,y, \lambda)=(x,y, \lambda |f(x,y) \overline{g(x,y)}|^{1/r})$. Then $q$ is the inverse map of $p$ and $\LF \setminus \Lp$ is diffeomorphic to $Q$.

Notice that diagram \eqref{dgcm} can also be seen as $\cP$ being the pull-back of the cyclic covering $-\rho_r$ by $\mpfg$, then $\cP$ is itself a cyclic covering of $r$ leaves. 

Now, let $(x,y) \in \Lfg$, then $\pr{\cP}(x,y)=\{(x,y,0) \in \LF \}$, \ie $\pr{\cP}(x,y)$ consists of only one point. Then $\cP$ from $\LF$ to $\Stres$ is a branched cyclic $r$-covering with ramification locus $\Lfg$. 

Statement \ref{3}) can be proved in the following way. Let $K$ be a connected component of $\Lfg$. Since $\mpfg$ is an open book fibration of $\Stres$, there exists a small closed tubular neighbourhood $\mc{U}$ of $K$ and a trivialisation $\delta \colon \mc{U} \to (\Sp{1} \times \D^2)$ such that $\delta(K)= \Sp{1} \times \{0\}$ and the following diagram commutes:
\begin{equation}
\xymatrix{
\mc{U} \setminus K \ar[dr]_{\mpfg} \ar[rr]^{\delta} &  & \Sp{1} \times \bigl(\D^2 \setminus \{0\}\bigr)   \ar[dl]^g \\
& \Sp{1} & \\
}
\end{equation}
where $\ensuremath{g(\lambda, w)= \frac{w}{|w|}}$ with $\lambda \in \Sp{1}$ and $w \in \D^2$.

Let $K'=\pr{\cP}(K)$ and $\calm{V}=\pr{\cP}(\mc{U})$. Let $\rho'_r \colon (\Sp{1} \times \D^2) \to (\Sp{1} \times \D^2)$ be the map defined by 
\begin{equation*}
\rho'_r(\lambda, w)=(-\lambda, w^r) \ ,
\end{equation*}
with $\lambda \in \Sp{1}$ and $w \in \D^2$.

The composition $\delta \circ \cP$ gives a cyclic branched $r$-covering of $\Sp{1} \times \D^2$ and in the other hand $-\rho'_r$ gives also a cyclic branched $r$-covering of $\Sp{1} \times \D^2$, both with the same ramification locus; then there exists an unique diffeomorphism $\delta' \colon \calm{V} \to \Sp{1} \times \D^2$ such that $\delta \circ \cP = -\rho'_r \circ \delta'$.

Then the following diagram commutes:
\begin{equation}\label{varr}
\xymatrix{
\calm{V} \setminus K' \ar[dr]_{\mpp} \ar[rr]^{\delta'} & & \Sp{1} \times (\D^2 \setminus \{0\}) \ar[dl]^{g} \\
& \Sp{1} & \\
}
\end{equation}
and $\mpp$ is an open book fibration of $\LF$. 

Now let us prove \ref{4}). 

Let $\F_{\fbg}$ be one Milnor fibre of $\mpfg$ and let $\Fp \subset \pr{\cP}(\F_{\fbg})$ be one fibre of $\mpp$. Since the diagram \eqref{dgcm} is a pull-back diagram, the restriction
\begin{equation*}
\cP|_{\Fp} \colon \Fp \to \F_{\fbg}
\end{equation*}
is a diffeomorphism. Moreover, the preimage $\pr{\cP}(\F_{\fbg})$ is the disjoint union of $r$ fibres of $\mpp$.

Let 
\begin{equation*}
\gamma \colon (\Stres \setminus \Lfg) \times \R \to (\Stres \setminus \Lfg)
\end{equation*}
be the flow of a vector field which is a lifting by $\mpfg$ of the canonical tangent vector field on $\Sp{1}$. Then $\gamma$ is transverse to the fibres of $\mpfg$ and a representative $h$ of the monodromy of $\mpfg$ is defined as the diffeomorphism of first return of $\gamma$ over the fibre $\F_{\fbg}$. Let
\begin{equation*}
\gp \colon (\LF \setminus \Lp) \times \R \to (\LF \setminus \Lp)
\end{equation*}
be a flow such that for all $\left((x,y,z),t \right) \in (\LF \setminus \Lp) \times \R$,
\begin{equation*}
\cP\left(\gp\left((x,y,z),t \right)\right) = \gamma\left( \cP(x,y,z),t \right) \ ,
\end{equation*}
and let $\hp$ be the representative of the monodromy of $\mpp$ defined as the diffeomorphism of first return of the flow $\gp$ on $\Fp$. By construction and since the covering is $r$-cyclic, one has the following commutative diagram
\begin{equation*}
\xymatrix{
\Fp \ar[d]_{h'} \ar[r]^{\cP} & \F_{\fbg} \ar[d]^{\mtp{h}{r}} \\
\Fp \ar[r]^{\cP} & \F_{\fbg} \\
}
\end{equation*}
\end{proof}

\section{The link $L_F$ as a plumbing link}
In this section we describe $L_F$ as a plumbing link and moreover, together with the open book structure given in the last section, we obtain that $L_F$ is a fibred plumbing link.

Let $f,g$ holomorphic functions from $\C^2$ to $\C$ such that the real analytic germ $\fbg \colon (\C^2,0) \to (\C,0)$ has an isolated singularity at the origin. Let $\mc{U}$ be a neighbourhood of the origin in $\C^2$, let $\pi \colon \mc{W} \to \mc{U}$ be a resolution of the holomorphic germ $fg$ and let $\Gamma$ be its dual graph. Then one has
\begin{equation*}
M_{\Gamma} \leftexp{t}{(m_1^f,\cdots,m_s^f)} + \leftexp{t}{b(L_f)}=0 \quad \text{and} \quad
M_{\Gamma} \leftexp{t}{(m_1^g,\cdots,m_s^g)} + \leftexp{t}{b(L_g)}=0 \ .
\end{equation*}

Now let us denote the link $L = L_f \cup - L_g$ by $L_f - L_g$. Then
\begin{equation*}
\pr{(M_{\Gamma})} \leftexp{t}{b( L_f - L_g)} = \pr{(M_{\Gamma})} \leftexp{t}{(b(L_f)-b(L_g))} \ .
\end{equation*}
Therefore
\begin{equation*}
\pr{(M_{\Gamma})} \leftexp{t}{b( L_f - L_g)} = - \leftexp{t}{(m^f_1-m^g_1,\ldots,m^f_s-m^g_s)} \ .
\end{equation*}
Hence $(m_1^f-m_1^g, \ldots, m_s^f - m_s^g)$ is the solution of the monodromical system of $L_f - L_g$. 

Let $\F_{\fbg}$ be the Milnor fibre of $\fbg$. Applying Theorem \ref{fibration} one obtains: 

\begin{thm}[{\cite[Cor.~2.2]{MR2115674}}] \label{fib2} 
The link $L_{f\bar{g}} = L_f - L_g$ is fibred if and only if for each node $v_j$ of \ $\Gamma$ one has $m_j^f \neq m_j^g$. Moreover, if this condition holds, then the quasi-periodic representative of the monodromy $h$ of $L_f-L_g$ has order $|m_j^f  - m_j^g |$ on $V_j \cap \F_{\fbg}$, where $V_j$ is the intersection of $\Sp{3}$ and the $\D^2$-bundle corresponding to $v_j$. 
\end{thm}

Hence, let $\Gamma_{\fbg}$ be the plumbing tree $\Gamma$ but with each vertex $i$ weighted by $m_i = m^f_i - m^g_i$. We have that the order of the monodromy $h$ of the Milnor fibration of $\fbg$ on $V_i \cap \F_{\fbg}$ equals $m_i = |m_i^f - m_i^g|$. Moreover, when $m_i^f - m_i^g <0$, the natural orientation of the $\Sp{1}$-fibres of $V_i$ is opposite to the one obtained by lifting the orientation of the circle $\Sp{1}$ with the Milnor fibration. In order to get the right orientation, one must change the orientation of the fibres on each $V_i$ where $m_i^f - m_i^g <0$. It implies that each edge joining two vertices $v_i$ and $v_j$ such that $m^f_i - m^g_i <0$ and $m^f_j - m^g_j \geq 0 $ is now weighted by $\varepsilon = -1$.

On the other hand, since $L_F$ is a graph manifold and $L'$ as the preimage of $\Lfg$ is union of $\Sp{1}$-fibres, we obtain that the pair $(L_F, L')$ is a plumbing link, but moreover, by Proposition~\ref{diagc}, $(L_F, L')$ is a fibred plumbing link.

\subsection{The Nielsen graph of the monodromy $h$}
As the monodromy $h$ of the Milnor fibration $\mpfg$ is a quasi-periodic diffeomorphism, we can describe it trough its Nielsen graph $\Ng{h}$.

Given a minimal reduction system $\cc$ for $h$, by Section~\ref{ngmh} we have that the Nielsen graph is given by Figure~\ref{fig:ngcph}, where $\Ng{h_{i}}$ is the Nielsen graph of the periodic diffeomorphism $h_i$.

\begin{figure}[H]
\begin{center}
\begin{tikzpicture}[xscale=1.2,yscale=1]
\filldraw[black] (0cm,0cm) circle (2.8pt);
\draw (-1cm,.7cm) -- (0cm,0cm) (5cm,0cm) -- (6cm,.7cm) (-1cm,-.7cm) -- (0cm,0cm) (6cm,-.7cm) -- (5cm,0cm);
\draw[line width=.69pt] (0cm,0cm) -- (5cm,0cm); 
\draw[thick,dotted] (-1cm,.7cm) -- (-1.5cm,1.05cm) (-1cm,-.7cm) -- (-1.5cm,-1.05cm);
\filldraw[black] (5cm,0cm) circle (2.8pt);
\draw[thick,dotted] (6cm,.7cm) -- (6.5cm,1.05cm) (6cm,-.7cm) -- (6.5cm,-1.05cm);
\draw (.8cm,.3cm) node {$(\lambda_i,\sigma_i)$}
(0cm,-.3cm) node {$q_{i}$}
(-.9cm,0cm) node {$[m_{i},g_{i}]$}
(4.2cm,.3cm) node {$(\lambda_j,\sigma_j)$}
(5cm,-.3cm) node {$q_{j}$}
(5.9cm,0cm) node {$[m_{j},g_{j}]$}
(2.5cm,-.3cm) node {$t$};
\draw[thick,loosely dotted] (-1.7cm,.3cm) arc (170:190:2cm) (6.7cm,-.3cm) arc (-10:10:2cm);
\end{tikzpicture}
\end{center}
\caption{Joining the Nielsen graphs $\Ng{h_{i}}$ and $\Ng{h_{j}}$.}
\label{fig:ngcph}
\end{figure}
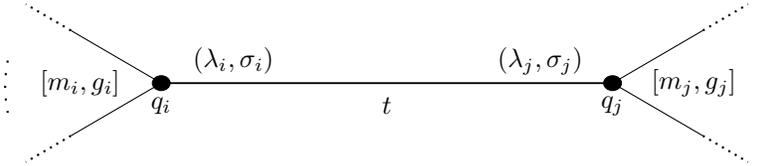

\subsection{The Nielsen graph of the diffeomorphism $\mtp{h}{r}$.}
Let $h \colon \F_{\fbg} \to \F_{\fbg}$ be the monodromy of the Milnor fibration $\mpfg$ with order $m$ and let $r \geq 2$. By Proposition~\ref{diagc}, the diffeomorphism $\mtp{h}{r}$ is a representative of the monodromy of the open book fibration $\mpp$, and by Section~\ref{ngmhr} we can describe its Nielsen invariants in terms of the Nielsen invariants of $h$.

Hence, by Lemmas~\ref{nihr} and \ref{twhr}, the Nielsen invariants of $\mtp{h}{r}$ are given by:
\begin{itemize}
\item The order of $\mtp{h}{r}$ is $\ensuremath{\otp{r}}= \frac{m}{\gcd(m,r)}$.
\item The genus of $\ostp{r}$ is 
\begin{equation*}
\gtp{r} =(g-1) \gcd(m,r)+1+\frac{1}{2} \sum_{i=1}^{s'} (\gcd(m,r)-\gcd(m/\lambda_i,r)) \ .
\end{equation*}
\item[-] The orbit space $\ostp{r}$ has a maximum of $\ensuremath{\sum_{i=1}^{s} \gcd(m/\lambda_i,r)}$ exceptional orbits and $\ensuremath{\sum_{i=s+1}^{s'} \gcd(m/\lambda_i,r)}$ boundary curves. In any case, the valency $(\ltp{r}_i, \stp{r}_i)$ is given by:
\begin{equation*}
\ltp{r}_i = \frac{m}{\lambda_i \gcd(m/\lambda_i,r)} \quad \text{and} \quad \stp{r} \times \frac{r}{\gcd(m,r)} \equiv \sigma_i \pmod{\ltp{r}} \ .
\end{equation*}
\item the twist $\ttp{r}=rt$.
\end{itemize}

Now, let $T(\mtp{h}{r})$ be the mapping torus of the diffeomorphism $\mtp{h}{r}$ and let $\F'$ be a fibre of $\mpp$. Hence $T(\mtp{h}{r})$ has boundary $\partial \F' \times \Sp{1}$. Let $\lambda \in \Sp{1}$, then $\partial \F' \times \{\lambda\}$ is the boundary of a manifold $\F'_{\lambda} \subset T(\mtp{h}{r})$ diffeomorphic to $\F'$.
Let
\begin{equation*}
\Op{\mtp{h}{r}}= T(\mtp{h}{r}) \bigcup_{\partial \F' \times \Sp{1}} (\partial \F' \times \D^{2}) \ ,
\end{equation*}
let $L(\mtp{h}{r})= (\partial \F' \times \{0\}) \subset (\partial \F' \times \D^{2})$ and let
\begin{equation*}
\pi_{\mtp{h}{r}} \colon (\Op{\mtp{h}{r}} \setminus L(\mtp{h}{r})) \to \Sp{1}
\end{equation*}
be such that $\pi_{\mtp{h}{r}} \left(\left(\partial \F' \times (0,\lambda]\right) \bigcup_{\{\partial \F' \times \lambda\}} \F'_{\lambda}\right)= \lambda$.

Notice that the pair $(\Op{\mtp{h}{r}}), L(\mtp{h}{r}))$ is diffeomorphic to $(L_F, L')$ and $\pi_{h^r}$  is the fibration $\mpp$. Thus, applying Theorem~\ref{mero}, the graph $\W(\Op{\mtp{h}{r}},L(\mtp{h}{r}))$ can be computed using the Nielsen graph $\Ng{\mtp{h}{r}}$. Moreover, we can compute the plumbing tree corresponding to $\LF$ from the graph $\W(\Op{\mtp{h}{r}},L(\mtp{h}{r}))$, giving in this way a description of $\LF$ as a graph manifold in terms of the link $\Lfg$ and the monodromy $h$ of the Milnor fibration $\mpfg$.

\section{Join Theorem for $d$-regular functions}
In this section we show that if $f$ and $g$ are holomorphic germs from $\C^2$ to $\C$ such that the real analytic germ $\fbg$ has isolated critical point and $r\geq 2$ is an integer, then the function $F  =\fbg + z^ r$ has a Milnor fibration with projection map $F/|F|$. Then we use a result given in \cite{MR0488073} to study the topology of this fibration.

\subsection{Cyclic suspensions}\label{cs}
In order to achieve our goal, we must define the cyclic suspension of a knot. For that we need some results of \cite{MR656605}, \cite{MR0488073} and \cite{MR0358797}.

\begin{defin}
A \textbf{knot} $\mc{K}= \kn{k}{K}$ is an oriented $k$-sphere with an oriented $2$-codimensional submanifold $K \subset \Sp{k}$. Notice that we use knot to denote also a link in the sense of Section~\ref{prel}.
\end{defin}

\begin{defin}[{\cite[Definition~1.1]{MR0488073}}]
A knot $\mc{L}= \kn{n}{L}$ is \defi{fibred} if $L$ is the binding of an open book fibration of $\Sp{n}$.
\end{defin}

\begin{defin}[{\cite[Definition~1.1]{MR0488073}}]
Let $M$ be a closed compact manifold. An \textbf{open book structure} for $M$ is a map $b \colon M \to \D^{2}$ such that zero is a regular value and
\begin{equation*}
\phi_b= \frac{b}{|b|} \colon M \setminus \pr{b}(0) \to \Sp{1} 
\end{equation*}
is a locally trivial fibration.
\end{defin}

Thus, given a fibred knot $\mc{L}= \kn{n}{L}$, we have associated an open book structure $b \colon \Sp{n} \to \D^{2}$ (see \cite[\S~1]{MR0488073}).

Notice that a fibred knot $\mc{L}=\kn{3}{L}$ is a fibred plumbing link, when $\Sp{3}$ is considered as a plumbing manifold and $L$ is a union of Seifert fibres, 

Let $\kn{m}{K}$ be a fibred knot with $m \geq 3$ and let $b \colon \Sp{m} \to \D^{2}$ be the associated open book structure. Let $\rho_r \colon \D^2 \to \D^2$ be the $r$-branched cyclic covering of $\D^2$ given by $\rho_r(z)=z^r$ with $0 \in \D^2$ as ramification locus.

The following result is a consequence of {\cite[Theorem~2.2]{MR0488073}}.

\begin{lem}\label{brcvkn}
The pull-back $\pi_r \colon \br\kn{m}{K} \to \Sp{m}$ of $\rho_r \colon \D^2 \to \D^2$ by the map $b$:
\begin{equation*}
\xymatrix{
\br\kn{m}{K} \ar[d]_{\pi_r} \ar[r] & \D^2 \ar[d]^{\rho_r} \\
\Sp{m} \ar[r]^{b} & \D^2
}
\end{equation*}
is an $r$-branched cyclic covering of \ $\Sp{m}$, branched along $K$.
\end{lem}

\begin{lem}[{\cite[Lemma~2.4]{MR0488073}}]\label{embij}
Let $V_1 \subset V_2 \subset M$ be proper embeddings of smooth manifolds of dimension $m-2$, $m$ and $m+2$ respectively. Assume that $M$ and $V_2$ are $2$-connected. Let $i \colon V_2 \to M$ be the standard inclusion. Then there exists an embedding of pairs $j \colon (V_2,V_1) \to (M, V_2)$ such that
\begin{enumerate}[a)]
\item the image $j(V_2) \subset M$ is transverse to $V_2$ with $j(V_2) \cap i(V_2)=V_1$,\label{adej}
\item the map $j$ is isotopic to $i$ through maps satisfying condition \ref{adej}) except at the end of the isotopy,\label{bdej}
\item the map $j$ is unique up to isotopy through maps which satisfy conditions \ref{adej}) and \ref{bdej}).
\end{enumerate}
\end{lem}

\begin{cor}
Let $\kn{m}{K}$ be a fibred knot and let $\kn{m+2}{\Sp{m}}$ be the trivial knot. Then the following diagram commutes:
\begin{equation*}
\xymatrix{
\br\kn{m}{K} \ar[d]_{\pi_r} \ar[r]^-{\hat{j}} & \br\kn{m+2}{\Sp{m}} \ar[d]^{\pi'_r} \\
\Sp{m} \ar[r]^j & \Sp{m+2}
}
\end{equation*}
where $j$ is the embedding of Lemma~\ref{embij}.
\end{cor}

\begin{defin}[{\cite[page 377]{MR0488073}}]
Let $K \otimes \br$ be the image 
\begin{equation*}
K \otimes \br = \hat{j}(\br\kn{m}{K}) \subset \br\kn{m+2}{\Sp{m}} \cong \Sp{m+2} \ .
\end{equation*}
Then we obtain a new knot $\kn{m+2}{K \otimes \br}$. This knot is called the \defi{$r$-fold cyclic suspension} or briefly \defi{$r$-cyclic suspension} of the knot $\kn{m}{K}$.
\end{defin}

\subsection{The property of $d$-regularity}\label{pdr}
Now we present the concept of $d$-regularity, which is a necessary and sufficient condition for a real analytic function $f$ to have Milnor fibration given by $\Phi_f=\frac{f}{||f||}$. Roughly speaking, a real analytic function $f$ with isolated singularity is $d$-regular if and only if $f$ has Milnor fibration $\Phi_f = f / ||f||$. Thus the link of the singularity $L_f$ is a fibred knot.

Let $U$ be an open neighbourhood of $0 \in \R^n$ with $n > 1$. Let $k \leq n$ and let $f \colon (U,0) \to (\R^k,0)$ be an analytic map defined on $U$ with isolated critical point at $0$.

Let $V=\pr{f}(0)$ and let $\B{}_{\varepsilon}$ be a closed ball in $\R^n$, centred at $0$, of sufficiently small radius $\varepsilon$, so that every sphere in this ball, centred at $0$, meets transversely $V$, if $V$ is not an isolated point at the origin.

One can define a family of real analytic spaces as follows.

For each $\ell \in \rp{k-1}$, consider the line $\calm{L}_{\ell} \subset \R^k$ passing through the origin corresponding to $\ell$, let $f|_{\B{}_{\varepsilon}}$ be the restriction of $f$ to the ball $\B{}_{\varepsilon}$ and set $X_{\ell} = \pr{f|_{\B{}_{\varepsilon}}}(\calm{L}_{\ell})$. 

Let $\calm L_{\ell}^\perp$ be the hyperplane orthogonal to $\calm L_{\ell}$ and let $\pi_{\ell} \colon \R^k \to \calm{L}_{\ell}^\perp$ be the orthogonal projection. Set $h_{\ell} = \pi_{\ell} \circ f|_{\B{}_{\varepsilon}}$, then $X_\ell$ is the vanishing set of $h_\ell$, which is real analytic. Hence $\{X_{\ell} \}$ is a family of real analytic hypersurfaces parametrised by $\rp{k-1}$.

The set of critical points of $h_{\ell}$ is contained in the set of critical points of $f|_{\B{}_{\varepsilon}}$ (see \cite[Lemma~2.1]{MR2647448}). As $f$ has isolated critical point at the origin, $h_{\ell}$ has the origin as its only critical point.

\begin{defin}[{\cite[Definition~2.3]{MR2647448}}]\label{canpen}
The family $\{ X_\ell \mid \ell \in \rp{k-1} \}$ is called the \defi{canonical pencil} of $f$.
\end{defin}

\begin{defin}[{\cite[Definition~2.4]{MR2647448}}]\label{dreg}
The map $f$ is said to be \defi{$d$-regular} at $0$ if there exists a metric $\mu$ induced by some positive definite quadratic form and there exists $\varepsilon' > 0$ such that every sphere (for the metric $\mu$) of radius $\varepsilon \leq \varepsilon'$ centred at $0$ meets every $X_\ell \setminus \pr{f}(0)$ transversely whenever the intersection is not empty.
\end{defin}

The following result follows from \cite[Th.~5.3]{MR2647448}, \cite[Cor.~5.4]{MR2647448} and the proof of \cite[Th. 11.2]{milnor:singular}.

\begin{thm}[Fibration Theorem]\label{fibthdr} 
Let $f \colon (\R^n,0) \to (\R^k,0)$ be a real analytic germ with isolated critical point. Then $f$ is $d$-regular if and only if the map
\begin{equation*}
\phi_f = \frac{f}{\norm{f}} \colon \Sp{n-1}_{\varepsilon} \to \Sp{1}
\end{equation*}
is a locally trivial fibration.
\end{thm}

\begin{rmk}
In fact, Theorem~\ref{fibthdr} appears in \cite{MR2647448} in more generality than here. It holds as well for $d$-regular real analytic functions with isolated critical value, which satisfy the Thom $a_f$-condition.
\end{rmk}

Next theorem is given as \cite[Theorem~1]{MR0358797} and the proof is essentially the same we present here. The proof is included here for clarity since there are involved other results that also differ slightly from the complex case.

\begin{thm}\label{rsuskn}
Let $f \colon (\R^n,0) \to (\R^2,0)$ be a $d$-regular real analytic germ with isolated critical point at the origin. Let $\varepsilon > 0$ be small enough and let $L_f= \pr{f}(0) \cap \Sp{n-1}$ be the link of the singularity at the origin. Let $F \colon (\R^n \times \C,0) \cong (\R^{n+2},0) \to (\R^2,0)$ be the map defined by 
\begin{equation*}
F(x_1, \ldots, x_n, z)= f(x_1, \ldots, x_n)+z^r \ ,
\end{equation*}
and denote by $\LF= \pr{F}(0) \cap \Sp{n+1}$ its link at the origin. Then the pair $\kn{n+1}{\LF}$ is the  $r$-fold cyclic suspension of the knot $\kn{n-1}{L_f}$.
\end{thm}
In the proof of this theorem we need the two following results. The first one gives a characterisation of a $d$-regular function. The second one is a generalisation of \cite[Lemma~4.1]{MR0488073} and the proof presented here is basically the same presented there.

For more clarity, from now on we denote a point $(x_1, \ldots, x_n) \in \R^n$ by $x$ and the vector $(x_1, ...,x_n) \in T_x \R^n \cong \R$ by $\vc{x}$.

\begin{lem}[{\cite[Lemma~5.2]{MR2647448}}]\label{vdr}
Let $U$ be an open neighbourhood of $0 \in \R^n$ with $n > 1$. Let $k \leq n$ and let $f \colon (U,0) \to (\R^k,0)$ be an analytic map defined on $U$ with isolated critical point at $0$. The map $f$ is $d$-regular, if and only if there exists a smooth vector field $\tilde{v}$ on $\B{}_{\varepsilon} \setminus \pr{f}(0)$ which has the following properties:
\begin{enumerate}[i)]
\item It is radial; \ie it is transverse to all spheres in $\B{}_{\varepsilon}$ centred at $0$ pointing outwards;  
\item it is tangent to each $X_{\ell} \setminus \pr{f}(0)$, whenever it is not empty;
\item it is transverse to all the tubes $\pr{f}(\partial \D_{\delta})$.
\end{enumerate}
\end{lem}

\begin{rmk}\label{conddr}
Notice that in Lemma~\ref{vdr}, the first condition is equivalent to ask that
\begin{equation}\label{it:pr1}
\langle \tilde{v}(x), \vc{x} \rangle_{\R^n} > 0 \quad \text{for all} \ \ x \in \B{}_{\varepsilon} \setminus \pr{f}(0) \ ,
\end{equation}
where $\langle \cdot, \cdot \rangle_{\R^n}$ is the Euclidean product in $\R^n$.

Moreover, the vector field $\tilde{v}$ can be adjusted, multiplying it by a positive real function such that the second and third conditions can be interpreted as
\begin{equation}\label{it:pr2}
D_x f (\tilde{v}(x)) = \vc{f(x)} \ .
\end{equation}
\end{rmk}

\begin{lem}\label{vbmo}
Let $f \colon (\R^n,0) \to (\R^2,0)$ be a $d$-regular real analytic germ with isolated critical point at the origin. For sufficiently small $\varepsilon > 0$, there exists a smooth vector field $v$ on $\B{}_{\varepsilon} \setminus \{0\}$ which satisfies:
\begin{enumerate}[I)]
\item $v$ lies over the radial vector field $w(x)=\vc{x}$ on $\C \cong \R^2$; \ie $v$ projects by $Df$ onto $\vc{x}$,
\item $\norm{x}$ increases along trajectories of $v$.
\end{enumerate}
\end{lem}

This result is given as \cite[Lemma~4.1]{MR0488073} for a complex polynomial with isolated singularity. The proof of Lemma~\ref{vbmo} is the same as proof of \cite[Lemma~4.1]{MR0488073} if we substitute the complex polynomial by $f$, and the vector field obtained by \cite[Lemma~5.9]{milnor:singular} by the vector field given by Lemma~\ref{vdr} and Remark~\ref{conddr}.
%
\begin{proof}[Proof of Theorem~\ref{rsuskn}]
Let
\begin{equation*}
\Sp{n+1}_{\varepsilon} = \{(x,z) \in \R^n \times \C \cong \R^{n+2} \mid \norm{(x,z)} = \varepsilon \} \ ,
\end{equation*}
and let
\begin{equation*}
\Sp{}_{\varepsilon}(t) = \{(x,z) \in \Sp{n+1}_{\varepsilon} \mid tf(x)+z=0 \} \quad \text{for} \quad 0 \leq t \leq 1 \ .
\end{equation*}
For $\varepsilon$ small and any $t$, $\Sp{}_{\varepsilon}(t)$ is the intersection of the sphere $\Sp{n+1}_{\varepsilon}$ and the smooth hypersurface $tf(x)+z=0$, thus $\Sp{}_{\varepsilon}(t)$ is a $(n-1)$-sphere. Hence $\{\Sp{}_{\varepsilon}(t)\}_{0 \leq t \leq 1}$ gives an isotopy between the standard sphere $\Sp{}_{\varepsilon}(0) \cong \Sp{n+1}_{\varepsilon}$ and $\Sp{}_{\varepsilon}(1)$.

Also $\Sp{}_{\varepsilon}(t)$ intersects transversely $\Sp{}_{\varepsilon}(1)$ and $\Sp{}_{\varepsilon}(t) \cap \Sp{}_{\varepsilon}(1) = L_f$ for all $t < 1$.

Take $\Sp{}_{\varepsilon}(1) \subset \Sp{n+1}_{\varepsilon}$ as the ``standard embedding''  and let
\begin{equation*}
\left(\Sp{}_{\varepsilon}(0),L_f \right) \subset \left(\Sp{n+1}_{\varepsilon}, \Sp{}_{\varepsilon}(0) \right)
\end{equation*}
be the embedding $j$ of Lemma~\ref{embij}. Let
\begin{equation*}
\overline{\mathbb{S}}^{n+1}_{\varepsilon, r} = \{ (x,z) \in \R^n \times \C \cong \R^{n+2} \mid x_1^2 + \cdots + x_n^2 + \va{z}^{2r} = \varepsilon^2 \}
\end{equation*}
and
\begin{equation*}
\overline{L}_F = \overline{\mathbb{S}}^{n+1}_{\varepsilon, r} \cap \pr{F}(0) \ .
\end{equation*}
Then the map $\pi \colon \overline{\mathbb{S}}^{n+1}_{\varepsilon, r} \to \Sp{n+1}_{\varepsilon}$ given by $\pi(x,z)= (x, z^r)$ gives a branched $r$-covering
\begin{equation*}
(\overline{\mathbb{S}}^{n+1}_{\varepsilon, r}, \overline{L}_F) \to (\Sp{n+1}_{\varepsilon}, \Sp{}_{\varepsilon}(1))
\end{equation*}
branched along $(\Sp{}_{\varepsilon}(0), L_f)$, and hence identifies $(\overline{\mathbb{S}}^{n+1}_{\varepsilon, r}, \overline{L}_F)$ as the $r$-fold cyclic suspension of $(\Sp{n-1}_{\varepsilon},L_f)$.

Thus, it just remains to show that the knot $(\overline{\mathbb{S}}^{n+1}_{\varepsilon, r}, \overline{L}_F)$ is diffeomorphic to the knot $(\Sp{n+1}_{\varepsilon}, \LF)$. This is done by pushing the pair $(\Sp{n+1}_{\varepsilon}, \LF)$ out to the other knot along a vector field defined in a small ball $(\B{n+2}_{\varepsilon'} \setminus \{0\}) \subset \R^{n+2} \setminus \{0\}$. Such a vector field can be obtained as follows:

By Lemma~\ref{vbmo}, there is a vector field $v$ on a small ball $\B{n}_{\varepsilon'} \setminus \{0\}$ such that $v$ lies over the radial vector field on $\R^2$ and $\langle v(x),\vc{x} \rangle_{\R^n}$ has positive real part (see Remark~\ref{conddr}). Then the vector field we are looking for is given by $v_1(x,z)=(v(x),z/r)$ on $(\R^{n+2} \times \C) \setminus (\{0\} \times \C)$ and $v_2(x,z)=(0,z)$ in a thin neighbourhood of $\{0\} \times (\C \setminus 0)$; so pasting $v_1$ and $v_2$ with a partition of unity we obtain the required vector field $v$.
\end{proof}

The singularity at the origin  of $F$ is called a \textbf{suspension singularity} of type $\fbg +z^r$.

\subsection{Join Theorem for $d$-regular functions}\label{jtdrf}

The aim in this section is to describe the homotopy type of the Milnor fibre $\F$ of the germ $F=f + z^r$ in terms of the homotopy type of the Milnor fibre $\F_f$ of $f$, where $f \colon (\R^n,0) \to (\R^2)$ is a $d$-regular function. For this, we first show that the function $F=f+z^r$ is $d$-regular, which assures that it has Milnor fibration with projection $\frac{F}{\norm{F}}$.

\begin{prop}\label{Fdr}
Let $f \colon (\R^n,0) \to (\R^2,0)$ be a $d$-regular real analytic germ with an isolated critical point at the origin. Let $F \colon (\R^n \times \C,0) \cong (\R^{n+2}) \to (\R^2,0)$ be the map defined by $F(x, z)= f(x)+z^r$. Then $F$ is $d$-regular.
\end{prop}
\begin{proof}
Let $\varepsilon_1, \varepsilon_2 >0$ be such that $\Sp{n+1}_{\varepsilon_1}$ is a Milnor ball for $F$  and $\Sp{n-1}_{\varepsilon_2}$ is a Milnor ball for $f$. Let
\begin{equation*}
\varepsilon = \min \{\varepsilon_1, \varepsilon_2\}
\end{equation*}
and let us consider $\Sp{n+1}_{\varepsilon}$ and $\Sp{n-1}_{\varepsilon} \subset \Sp{n+1}_{\varepsilon}$, where the latter is defined by
\begin{equation*}
\Sp{n-1}_{\varepsilon} = \left\{(x, 0) \in \Sp{n+1}_{\varepsilon} \right\} \ .
\end{equation*}

By Lemma~\ref{vdr}, there exists a smooth vector field $v_1$ on $\B{n}_{\varepsilon} \setminus \pr{f}(0)$ such that
\begin{enumerate}[i)]
\item $\langle v_1(x), \vc{x} \rangle_{\R^n} > 0$,
\item $D_x f (v_1(x))= \vc{f(x)}$,
\end{enumerate}
for all $x \in \B{n}_{\varepsilon} \setminus \pr{f}(0)$.

On the other hand, let $\rho_r \colon \C \cong \R^2 \to \C \cong \R^2$ be the function defined by $\rho_r(z)=z^r$ and let $v_2$ be the radial vector field on $\C \cong \R^2$ defined by $v_2(z)= \frac{\vc{z}}{r}$, then
\begin{enumerate}[i)]
\item $\langle v_2(z), \vc{z} \rangle_{\R^2} > 0$,
\item $D_z \rho_r (v_2(z))= \vc{\rho_r{z}}$, 
\end{enumerate}
for all $z \in \C \setminus \{0\}$.

In order to prove that $F$ is $d$-regular, the idea is to find a vector field $v$ on $\B{n+2}_{\varepsilon} \setminus \pr{F}(0)$ such that
\begin{enumerate}[i)]
\item $\langle v(x,z), \vc{(x,z)} \rangle > 0$ for all $(x,z) \in \B{n+2}_{\varepsilon} \setminus \pr{F}(0)$,\label{v1}
\item $D_{(x,z)} F (v(x,z))= \vc{F(x,z)}$ for all $(x,z) \in \B{n+2}_{\varepsilon} \setminus \pr{F}(0)$.\label{v2}
\end{enumerate}
Let $v$ be the vector field on $\B{n+2}_{\varepsilon} \setminus \pr{F}(0)$ defined by $v(x,z)=(v_1(x),v_2(z))$. Then
\begin{equation*}
\langle v(x,z), \vc{(x,z)}\rangle_{\R^{n+2}} = \langle v_1(x), \vc{x}\rangle_{\R^n} + \langle v_2(z), \vc{z} \rangle_{\R^2} > 0 \ ,
\end{equation*}
since both terms are greater than zero. Then condition \eqref{v1} follows (see Figure~\ref{fig:vsv1v2}). We also have that
\begin{align*}
D_{x,z} F (v(x,z)) &= D_{x,z} F (v_1(x),v_2(z)) = D_{x,z} F (v_1(x),\vc{0}) + D_{x,z} F (\vc{0},v_2(z)) \\
                  &= D_x f (v_1(x)) + D_z \rho_r (v_2(z)) = \vc{f(x)} + \vc{\rho_r(z)} = \vc{F(x,z)} \ .
\end{align*}
And condition \eqref{v2} holds. By Lemma~\ref{vdr} and Remark~\ref{conddr}, $F$ is $d$-regular.
\end{proof}

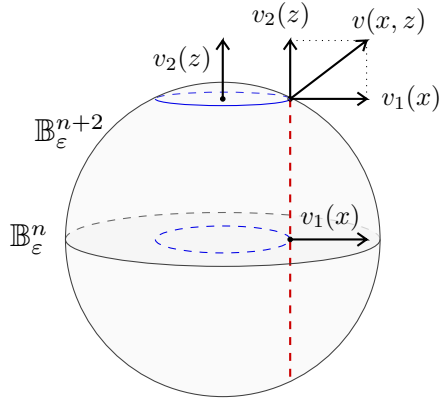
\begin{figure}[H]
\begin{center}
\begin{tikzpicture}[scale=1.7]
\draw[gray, dashed,opacity=.4] (35pt,0pt) arc (0:65:35pt and 6pt);
\draw[gray!50!black, dashed] (-35pt,0pt) arc (-180:-295:35pt and 6pt);
\draw[blue!90!black, dashed] (15pt,31.3pt) arc (0:180:15pt and 1.5pt);
\fill[gray!20!white, opacity=.2] (0pt,0pt) circle (35pt);
\fill[gray!30!white, opacity=.1] (0pt,0pt) ellipse (35pt and 6pt);
\draw[gray!50!black](0pt,0pt) circle (35pt); 
\draw[blue!90!black,dashed] (15pt,0pt) arc (0:180:15pt and 3pt);
\draw[red!80!black,thick,dashed] (15pt,31.2pt) -- (15pt,-31.2pt);
\draw[blue!90!black, dashed] (-15pt,0pt) arc (180:360:15pt and 3pt);
\draw[xshift=15pt, thick] (0pt,0pt) -- ++(0:17pt) -- +(145:2.5pt) (0pt,0pt) ++(0:17pt) -- +(215:2.5pt);
\draw[yshift=31.3pt, thick] (0pt,0pt) -- ++(90:13pt) -- +(235:2.5pt) (0pt,0pt) ++(90:13pt) -- +(305:2.5pt);
\filldraw (15pt,0pt) circle (.5pt) (0pt,31.3pt) circle (.5pt);
\draw[blue!90!black] (-15pt,31.3pt) arc (180:360:15pt and 1.5pt);
\draw[gray!50!black] (-35pt,0pt) arc (180:360:35pt and 6pt);
\begin{scope}[xshift=15pt, yshift=31.3pt]
\draw (0pt,0pt) circle (.5pt);
\draw[thick] (0pt,0pt) -- ++(0:17pt) -- +(145:2.5pt) (0pt,0pt) ++(0:17pt) -- +(215:2.5pt)
      (0pt,0pt) -- ++(90:13pt) -- +(235:2.5pt) (0pt,0pt) ++(90:13pt) -- +(305:2.5pt);
\draw[dotted] (17pt,0pt) -- (17pt,13pt) (0pt,13pt) -- (17pt,13pt);
\draw[thick] (0pt,0pt) -- ++(17pt,13pt) -- +(182.4:2.5pt) (0pt,0pt) ++(17pt,13pt) -- +(252.4:2.5pt);
\end{scope}
\draw (-34pt,24pt) node[scale=1.2] {$\B{n+2}_{\varepsilon}$}
      (-43pt,0pt) node[scale=1.2] {$\B{n}_{\varepsilon}$}
      (-9pt,40pt) node {$v_2(z)$}
      (13pt,50pt) node {$v_2(z)$}
      (24pt,5pt) node {$v_1(x)$}
      (42pt,31.3pt) node {$v_1(x)$}
      (37pt,48.5pt) node {$v(x,z)$};
\end{tikzpicture}
\end{center}
\caption{The vector field $v$ on $\B{n+2}_{\varepsilon} \setminus \pr{F}(0)$.}
\label{fig:vsv1v2}
\end{figure}

By Theorem~\ref{fibthdr} and Proposition~\ref{Fdr}, one has the following result.

\begin{cor}
Given $F$ as in Proposition~\ref{Fdr}, $F$ has Milnor fibration with projection
\begin{equation*}
\phi_F= \frac{F}{\norm{F}} \colon \Sp{5}_{\varepsilon} \setminus \LF \to \Sp{1} \ .
\end{equation*}
\end{cor}

\begin{rmk}
Let $\rho_r \colon \C \to \C$ be the function defined by $\rho_r(z)=z^r$. Note that $\rho_r$ has a Milnor fibration with projection 
\begin{equation*}
\phi_{\rho_r}=\frac{z^r}{\va{z}^r} \colon \Sp{1}_{\varepsilon} \to \Sp{1}\ ,
\end{equation*}
for any $\varepsilon > 0$ and the Milnor fibre $\F_{\rho_r}$ consists of $r$ points in $\Sp{1}_{\varepsilon}$ and the monodromy $h_{\rho_r}$ is given by a cyclic permutation of this $r$ points.
\end{rmk}

The following result is an adaptation of \cite[Lemma~6.1]{MR0488073} to the particular case of cyclic suspensions given by links of singularities and it describes the homotopy type of the Milnor fibre of $F$ in terms of the Milnor fibres of $f$ and $\rho_r$.

\begin{thm}\label{jointh}
Let $f \colon (\R^n,0) \to (\R^2,0)$ be a $d$-regular real analytic germ with an isolated critical point at the origin. Consider its Milnor fibration with projection map $\phi_f=\frac{f}{\norm{f}}$ and let $\F_f$ be its fibre and let $h_f$ be its monodromy. Let $F \colon (\R^n \times \C,0) \cong (\R^{n+2}) \to (\R^2,0)$ be the map defined by $F(x, z)= f(x)+z^r$ and let $\phi_F=\frac{F}{\norm{F}}$ be the projection map of its Milnor fibration; let $\F$ be its Milnor fibre and $h_F$ its monodromy. 

Then there exists a homotopy equivalence $\alpha \colon \F_f \ast \F_{\rho_r} \to \F$ which is compatible with the monodromy maps and their join; \ie the following diagram commutes
\begin{equation}
\xymatrix{
\F_f \ast \F_{\rho_r} \ar[d]_{\alpha} \ar[rr]^{h_f \ast h_{\rho_r}} && \F_f \ast \F_{\rho_r} \ar[d]^{\alpha} \\
\F \ar[rr]^{h_F} && \F \\
}
\end{equation}
where $h_f \ast h_{\rho_r} \colon \F_f \ast \F_{\rho_r} \to \F_f \ast \F_{\rho_r}$ is the map defined by
\begin{equation*}
h_f \ast h_{\rho_r}([x,t,y])= [h_f(x),t,h_{\rho_r}(y)] \ .
\end{equation*}
\end{thm}

\section{Open books given by the Milnor fibration $\Phi_F$}
In this section we recall the algorithm that allows us to compute the topology of the link $L_F$ and we use Theorem~\ref{jointh} in some examples, in order to see how is the open book decomposition of $\Sp{5}$ given by the Milnor fibration of $F$.

Let $f,g \colon (\C^2,0) \to (\C,0)$ be two complex analytic germs such that the real analytic germ $\fbg \colon (\C^2,0) \to (\C,0)$ has an isolated singularity at the origin. Let $F \colon (\C^3,0) \to (\C,0)$ be the real analytic germ defined by
\begin{equation*}
F(x,y,z) = \fbgplusr
\end{equation*}
with $r \in \Z{}^+$.

As we saw in Section \ref{Prim}, the following diagram commutes:
\begin{equation*}
\xymatrix{
\LF\setminus \Lp \ar[d]_{\mpp} \ar[r]^{\cP} & \Stres \setminus \Lfg \ar[d]^{\mpfg} \\
\Sp{1} \ar[r]^{\rho_r} & \Sp{1} \\
}
\end{equation*}
where $\Lfg$ is the link of $\fbg$, $\LF$ is the link of $F$, $\cP$ is the projection onto the first two coordinates,
$\Lp = \pr{\cP}(\Lfg)$, $\mpfg$ is the Milnor fibration of $\fbg$, $\mpp=\frac{z}{\mid z \mid}$ and $\rho_r(z)=z^r$.

With the tools given in the previous sections, we are able to give a description of the link $\LF$ as a graph manifold in terms of the link $\Lfg$ and the monodromy $h$ of the Milnor fibration $\mpfg$ as follows:

\begin{enumerate}
\item[\textbf{Step 1.}] We compute the plumbing tree $\Gamma_{\fbg}$ where each vertex $i$ is weighted by $m_i = m^f_i - m^g_i$ and change the negative multiplicities into positive ones.
\item[\textbf{Step 2.}] Using that the monodromy $h$ of the Milnor fibration $\Gamma_{\fbg}$ is a quasi-periodic diffeomorphism of the corresponding Milnor fibre $\F$, we compute the Nielsen graph $\Ng{h}$.
\item[\textbf{Step 3.}] We compute the Nielsen graph $\Ng{\mtp{h}{r}}$ of the diffeomorphism $\mtp{h}{r}$, which is a representative of the monodromy of the open book fibration $\mpp$. This allows us to describe $L_F$ as an open book.
\item[\textbf{Step 4.}] Applying Theorem~\ref{mero}, we compute the graph $\W(\Op{\mtp{h}{r}},L(\mtp{h}{r}))$ from the Nielsen graph $\Ng{\mtp{h}{r}}$, where $\Op{\mtp{h}{r}}=\LF$ and $L(\mtp{h}{r}))=\Lp$.
\item[\textbf{Step 5.}] We compute the plumbing tree $\Gamma$ such that $\LF= \partial P(\Gamma)$, where $P(\Gamma)$ is the plumbing manifold obtained from $\Gamma$.  
\end{enumerate}

\subsection{Examples}\label{exam}

In this section we show examples of functions $F$ such that the Milnor fibration gives an open book decomposition of $\Sp{5}$ that cannot come from complex singularities. To obtain this, first we describe the link $L_F$ as a graph manifold with the algorithm, then we apply (in some cases) Theorem~\ref{jointh} in order to describe the homotopy type of the corresponding Milnor fibre.

\begin{exa}
Let $f \colon (\C^2,0) \to (\C,0)$ be the complex analytic germ defined by $f(x,y)=x^2+y^7$ and let $g \colon (\C^2,0) \to (\C,0)$ be the complex analytic germ defined by $g(x,y)=x^5+y^2$. Let $F \colon (\C^3,0) \to (\C,0)$ be the germ defined by
\begin{equation*}
F(x,y,z) = (x^2+y^7)\overline{(x^5+y^2)} + z^3 \ .
\end{equation*}

\subsubsection*{First step:} 
The resolution graph $A(fg)$, where the multiplicity $m_i = m^f_i+m^g_i$ at the vertex $i$ appears as $\ensuremath{\binom{m^f_i}{m^g_i}}$ is given by
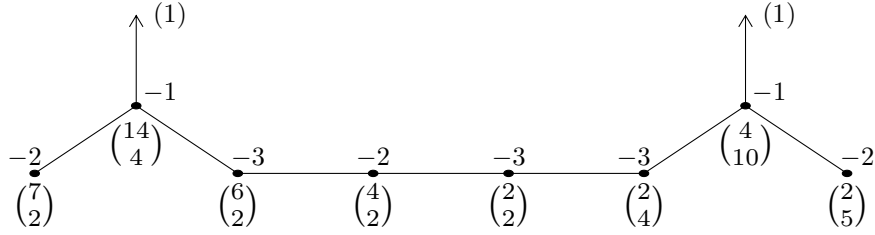
\begin{center}
\begin{tikzpicture}[xscale=.45,yscale=.3]
\draw[black] (0cm,0cm) -- (3cm,3cm) -- (6cm,0cm) -- (18cm,0cm) -- (21cm,3cm) -- (24cm,0cm) (3cm,3cm) -- (3cm,7cm) (21cm,3cm) -- (21cm,7cm) (2.8cm,6.5cm) -- (3cm,7cm) -- (3.2cm,6.5cm) (20.8cm,6.5cm) -- (21cm,7cm) -- (21.2cm,6.5cm);
\filldraw (0cm,0cm) circle (4pt) (3cm,3cm) circle (4pt) (6cm,0cm) circle (4pt) (10cm,0cm) circle (4pt) (14cm,0cm) circle (4pt) (18cm,0cm) circle (4pt) (21cm,3cm) circle (4pt) (24cm,0cm) circle (4pt);
\draw (-.3cm,.7cm) node {$-2$}
(3.7cm,3.6cm) node {$-1$}
(6.3cm,.7cm) node {$-3$}
(10cm,.7cm) node {$-2$}
(14cm,.7cm) node {$-3$}
(17.7cm,.7cm) node {$-3$}
(21.7,3.6cm) node {$-1$}
(24.3cm,.7cm) node {$-2$};
\draw (0cm,-1.4cm) node [scale=1.3] {$\binom{7}{2}$}
(3cm,1.3cm) node [scale=1.3] {$\binom{14}{4}$}
(6cm,-1.4cm) node [scale=1.3] {$\binom{6}{2}$}
(10cm,-1.4cm) node [scale=1.3] {$\binom{4}{2}$}
(14cm,-1.4cm) node [scale=1.3] {$\binom{2}{2}$}
(18cm,-1.4cm) node [scale=1.3] {$\binom{2}{4}$}
(21cm,1.3cm) node [scale=1.3] {$\binom{4}{10}$}
(24cm,-1.4cm) node [scale=1.3] {$\binom{2}{5}$}
(4cm,7cm) node {$(1)$}
(22cm,7cm) node {$(1)$};
\end{tikzpicture}
\end{center}
Then, the plumbing tree $\Gamma_{\fbg}$ is given by
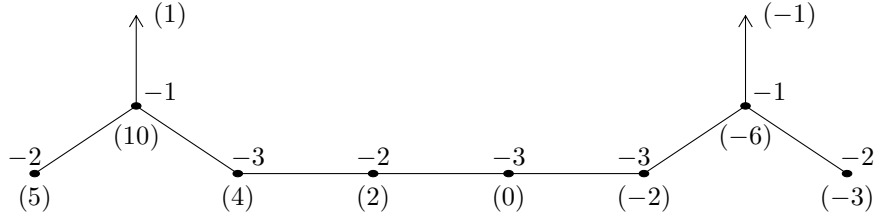
\begin{center}
\begin{tikzpicture}[xscale=.45,yscale=.3]
\draw[black] (0cm,0cm) -- (3cm,3cm) -- (6cm,0cm) -- (18cm,0cm) -- (21cm,3cm) -- (24cm,0cm) (3cm,3cm) -- (3cm,7cm) (21cm,3cm) -- (21cm,7cm) (2.8cm,6.5cm) -- (3cm,7cm) -- (3.2cm,6.5cm) (20.8cm,6.5cm) -- (21cm,7cm) -- (21.2cm,6.5cm);
\filldraw (0cm,0cm) circle (4pt) (3cm,3cm) circle (4pt) (6cm,0cm) circle (4pt) (10cm,0cm) circle (4pt) (14cm,0cm) circle (4pt) (18cm,0cm) circle (4pt) (21cm,3cm) circle (4pt) (24cm,0cm) circle (4pt);
\draw (-.3cm,.7cm) node {$-2$}
(3.7cm,3.6cm) node {$-1$}
(6.3cm,.7cm) node {$-3$}
(10cm,.7cm) node {$-2$}
(14cm,.7cm) node {$-3$}
(17.7cm,.7cm) node {$-3$}
(21.7,3.6cm) node {$-1$}
(24.3cm,.7cm) node {$-2$};
\draw (0cm,-1.1cm) node {$(5)$}
(3cm,1.6cm) node {$(10)$}
(6cm,-1.1cm) node {$(4)$}
(10cm,-1.1cm) node {$(2)$}
(14cm,-1.1cm) node {$(0)$}
(18cm,-1.1cm) node {$(-2)$}
(21cm,1.6cm) node {$(-6)$}
(24cm,-1.1cm) node {$(-3)$}
(4cm,7cm) node {$(1)$}
(22.3cm,7cm) node {$(-1)$};
\end{tikzpicture}
\end{center}
After the change of orientation in the Seifert fibres, we obtain the following graph:
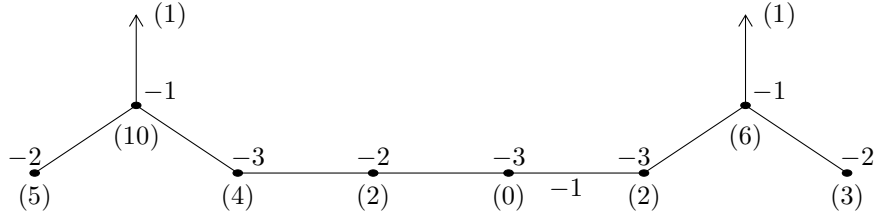
\begin{center}
\begin{tikzpicture}[xscale=.45,yscale=.3]
\draw[black] (0cm,0cm) -- (3cm,3cm) -- (6cm,0cm) -- (18cm,0cm) -- (21cm,3cm) -- (24cm,0cm) (3cm,3cm) -- (3cm,7cm) (21cm,3cm) -- (21cm,7cm) (2.8cm,6.5cm) -- (3cm,7cm) -- (3.2cm,6.5cm) (20.8cm,6.5cm) -- (21cm,7cm) -- (21.2cm,6.5cm);
\filldraw (0cm,0cm) circle (4pt) (3cm,3cm) circle (4pt) (6cm,0cm) circle (4pt) (10cm,0cm) circle (4pt) (14cm,0cm) circle (4pt) (18cm,0cm) circle (4pt) (21cm,3cm) circle (4pt) (24cm,0cm) circle (4pt);
\draw (-.3cm,.7cm) node {$-2$}
(3.7cm,3.6cm) node {$-1$}
(6.3cm,.7cm) node {$-3$}
(10cm,.7cm) node {$-2$}
(14cm,.7cm) node {$-3$}
(17.7cm,.7cm) node {$-3$}
(21.7,3.6cm) node {$-1$}
(24.3cm,.7cm) node {$-2$};
\draw (0cm,-1.1cm) node {$(5)$}
(3cm,1.6cm) node {$(10)$}
(6cm,-1.1cm) node {$(4)$}
(10cm,-1.1cm) node {$(2)$}
(14cm,-1.1cm) node {$(0)$}
(15.7cm,-.7cm) node {$-1$}
(18cm,-1.1cm) node {$(2)$}
(21cm,1.6cm) node {$(6)$}
(24cm,-1.1cm) node {$(3)$}
(4cm,7cm) node {$(1)$}
(22cm,7cm) node {$(1)$};
\end{tikzpicture}
\end{center}

\subsubsection*{Second step:}
From the last graph, we compute the Nielsen graph $\Ng{h}$ of the monodromy $h$ of the Milnor fibration $\mpfg$.
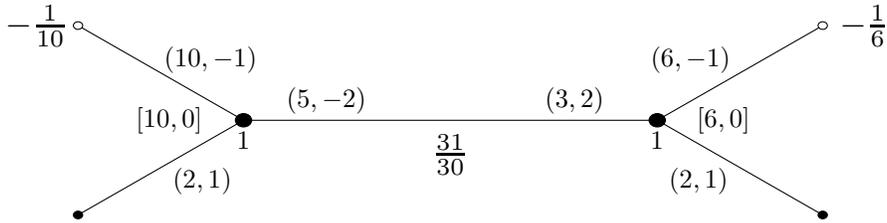
\begin{center}
\begin{tikzpicture}[xscale=1.1,yscale=.9]
\filldraw[black] (0cm,0cm) circle (2.8pt) (5cm,0cm) circle (2.8pt) (-2cm,-1.4cm) circle (1.6pt) (7cm,-1.4cm) circle (1.6pt);
\draw (-1.95cm,1.365cm) -- (0cm,0cm) --  (5cm,0cm) -- (6.95cm,1.365cm) (-2cm,-1.4cm) -- (0cm,0cm) (-2cm,1.4cm) circle (1.6pt) (7cm,-1.4cm) -- (5cm,0cm) (7cm,1.4cm) circle (1.6pt);
\draw (1cm,.3cm) node {$(5,-2)$}
(0cm,-.3cm) node {$1$}
(-.9cm,0cm) node {$[10,0]$}
(4cm,.3cm) node {$(3,2)$}
(5cm,-.3cm) node {$1$}
(5.8cm,0cm) node {$[6,0]$}
(-.4cm,.9cm) node {$(10,-1)$}
(-.5cm,-.9cm) node {$(2,1)$}
(5.4cm,.9cm) node {$(6,-1)$}
(5.5cm,-.9cm) node {$(2,1)$}
(2.5cm,-.5cm) node[scale=1.3] {$\frac{31}{30}$}
(-2.5cm,1.4cm) node [scale=1.3] {$-\frac{1}{10}$}
(7.5cm,1.4cm) node [scale=1.3] {$-\frac{1}{6}$}; 
\end{tikzpicture}
\end{center}

Here the twist $\frac{31}{30}$ can be computed in two ways: 
\begin{itemize}
\item The first one is to consider the partial twists between the nodes and get the twist as the sum of them,
\item the other one is to compute the $\alpha$ corresponding to the edge joining the nodes.
\end{itemize}
By~\cite[Th.~5.1]{NeuRay:plumb}, we have that
\begin{equation*}
\frac{\alpha}{\alpha - \beta} = 3 - \cfrac{1}{2 - \cfrac{1}{3 - \cfrac{1}{3}}}
\end{equation*}
then $\alpha=31$. Applying equation~\eqref{f3}, one obtains that
\begin{equation*}
t= \frac{\alpha}{m_j \lambda} = \frac{31}{(6)(5)} = \frac{31}{30} \ .
\end{equation*}

\subsubsection*{Third step:}
Let $F \colon (\C^3,0) \to (\C,0)$ be the real analytic germ defined by
\begin{equation*}
F(x,y,z) = (x^2+y^7)\overline{(x^5+y^2)} + z^3 \ .
\end{equation*}
The link $\LF$ has an open book fibration with binding $\Lp$ and monodromy $\mtp{h}{3}$. This monodromy is a quasi-periodic diffeomorphism, then we compute the Nielsen graph $\Ng{\mtp{h}{3}}$ from the graph $\Ng{h}$.
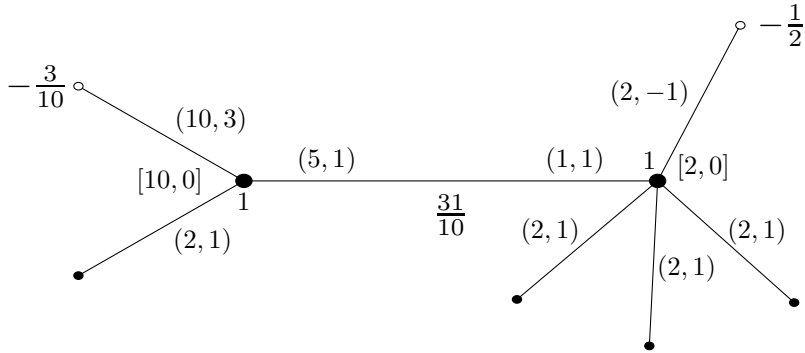
\begin{center}
\begin{tikzpicture}[xscale=1.1,yscale=.9]
\filldraw[black] (0cm,0cm) circle (2.8pt) (5cm,0cm) circle (2.8pt) (-2cm,-1.4cm) circle (1.6pt) (3.3cm,-1.752cm) circle (1.6pt) (4.9cm,-2.439cm) circle (1.6pt) (6.65cm,-1.799cm) circle (1.6pt);
\draw (-1.95cm,1.365cm) -- (0cm,0cm) --  (5cm,0cm) -- (5.975cm,2.25cm) (-2cm,-1.4cm) -- (0cm,0cm) (4.9cm,-2.439cm) -- (5cm,0cm) -- (3.3cm,-1.752cm) (6.65cm,-1.799cm) -- (5cm,0cm) (-2cm,1.4cm) circle (1.6pt) (6cm,2.3cm) circle (1.6pt);
\draw (1cm,.3cm) node {$(5,1)$}
(0cm,-.3cm) node {$1$}
(-.9cm,0cm) node {$[10,0]$}
(4cm,.3cm) node {$(1,1)$}
(4.9cm,.3cm) node {$1$}
(5.55cm,.2cm) node {$[2,0]$}
(-.4cm,.9cm) node {$(10,3)$}
(-.5cm,-.9cm) node {$(2,1)$}
(4.9cm,1.3cm) node {$(2,-1)$}
(5.35cm,-1.3cm) node {$(2,1)$}
(3.7cm,-.75cm) node {$(2,1)$}
(6.2cm,-.75cm) node {$(2,1)$}
(2.5cm,-.5cm) node[scale=1.3] {$\frac{31}{10}$}
(-2.5cm,1.4cm) node [scale=1.3] {$-\frac{3}{10}$}
(6.5cm,2.3cm) node [scale=1.3] {$-\frac{1}{2}$}; 
\end{tikzpicture}
\end{center}

\subsubsection*{Fourth step:}
Applying Theorem~\ref{mero}, we compute the graph $\W(\LF,\Lp)$ from the Nielsen graph $\Ng{\mtp{h}{3}}$:
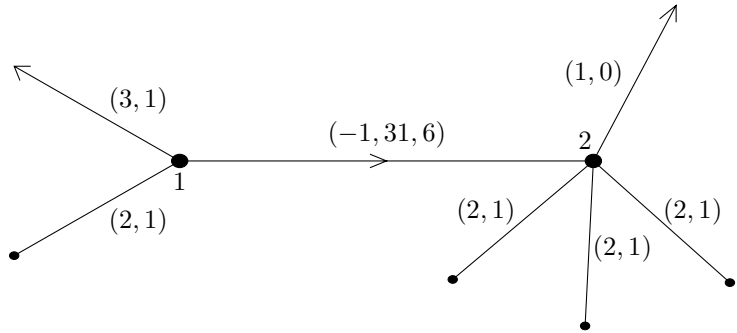
\begin{center}
\begin{tikzpicture}[xscale=1.1,yscale=.9]
\filldraw[black] (0cm,0cm) circle (2.8pt) (5cm,0cm) circle (2.8pt) (-2cm,-1.4cm) circle (1.6pt) (3.3cm,-1.752cm) circle (1.6pt) (4.9cm,-2.439cm) circle (1.6pt) (6.65cm,-1.799cm) circle (1.6pt);
\draw (-2cm,1.4cm) -- (0cm,0cm) --  (5cm,0cm) -- (6cm,2.3cm) (-2cm,-1.4cm) -- (0cm,0cm) (4.9cm,-2.439cm) -- (5cm,0cm) -- (3.3cm,-1.752cm) (6.65cm,-1.799cm) -- (5cm,0cm) (-1.93cm,1.189cm) -- (-2cm,1.4cm) -- (-1.8cm,1.4cm) (5.8cm,2.1cm) -- (6cm,2.3cm) -- (6cm,2.05cm) (2.3cm,-.1cm) -- (2.5cm,0cm) -- (2.3cm,.1cm);
\draw (2.5cm,.4cm) node {$(-1,31,6)$}
(0cm,-.3cm) node {$1$}
(4.9cm,.3cm) node {$2$}
(-.5cm,.9cm) node {$(3,1)$}
(-.5cm,-.9cm) node {$(2,1)$}
(5cm,1.3cm) node {$(1,0)$}
(5.35cm,-1.3cm) node {$(2,1)$}
(3.7cm,-.75cm) node {$(2,1)$}
(6.2cm,-.75cm) node {$(2,1)$};
\end{tikzpicture}
\end{center}

\subsubsection*{Fifth step:}
From the graph $\W(\LF,\Lp)$ we compute the corresponding plumbing tree. First, we have the following equations:
\begin{equation*}
\frac{3}{3-1} = 2- \frac{1}{2}  = [2,2] \ , \quad \ \frac{2}{2-1} = 2 = [2] \ , \quad \ \frac{31}{31-6} = [2,2,2,2,7] \ ,
\end{equation*}
where 
\begin{equation*}
[b_1, \ldots, b_k] = b_1 - \cfrac{1}{b_2 - \cfrac{1}{\ddots \cfrac{}{-\cfrac{1}{b_k}}}} \ .
\end{equation*}
Then the plumbing tree $\Gamma$ is given by
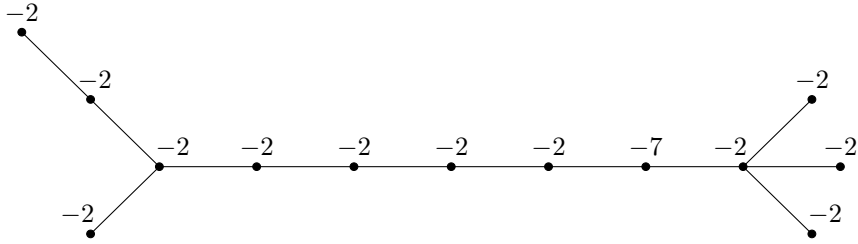
\begin{center}
\begin{tikzpicture}[xscale=.92,yscale=.9]
\draw (135:80pt) -- (0pt,0pt) -- (240pt,0pt) -- +(45:40pt) (240pt,0pt) -- (280pt,0pt) (240pt,0pt) -- +(-45:40pt) (0pt,0pt) -- +(225:40pt);
\filldraw (135:80pt) circle (1.6pt) node [above] {$-2$} 
(135:40pt) circle (1.6pt) node [above] {$\ -2$}
(0pt,0pt) circle (1.6pt) node [above] {$\quad -2$}
+(225:40pt) circle (1.6pt) node [above] {$-2 \quad$}
(40pt,0pt) circle (1.6pt) node [above] {$-2$}
(80pt,0pt) circle (1.6pt) node [above] {$-2$}
(120pt,0pt) circle (1.6pt) node [above] {$-2$}
(160pt,0pt) circle (1.6pt) node [above] {$-2$}
(200pt,0pt) circle (1.6pt) node [above] {$-7$}
(240pt,0pt) circle (1.6pt) node [above] {$-2 \quad$}
+(45:40pt) circle (1.6pt) node [above] {$-2$} 
+(-45:40pt) circle (1.6pt) node [above] {$\quad -2$}
(280pt,0pt) circle (1.6pt) node [above] {$-2$};
\end{tikzpicture}
\end{center}
where $\LF \cong \partial P(\Gamma)$ and $P(\Gamma)$ is the four-manifold obtained by plumbing $2$-discs bundles according to $\Gamma$.

The plumbing $P(\Gamma)$ given by the plumbing tree $\Gamma$ contains in its interior an exceptional divisor $E$ as a strong deformation retract. Then the divisor $E$ can be blown down to a point, and we get a complex surface $V_\Gamma$ with a normal singularity at $0$. As in the proof of \cite[Theorem~4]{MR2922705}, we compute the canonical class $K$ of $V_\Gamma$ and obtain that $K$ has non-integer coefficients.

It follows that the singularity $(V_\Gamma,0)$ is not numerically Gorenstein and therefore it is not Gorenstein. Then there is not a complex analytic germ $G \colon (\C^3,0) \to (\C,0)$ with isolated singularity at the origin such that the link $L_G$ is homeomorphic to the link $L_F$.
\end{exa}

\begin{exa}\label{exfeo}
Let $f \colon (\C^2,0) \to (\C,0)$ be the complex analytic germ defined by $f(x,y)=(x^2+y^3)$ and let $g \colon (\C^2,0) \to (\C,0)$ be the complex analytic germ defined by $g(x,y)=(x^3+y^2)$. 
Let $F \colon (\C^3,0) \to (\C,0)$ be the germ defined by
\begin{equation*}
F(x,y,z) = (x^2+y^3)\overline{(x^3+y^2)} + z^2 \ .
\end{equation*}

The plumbing tree $\Gamma_{\fbg}$ is given in Figure  \ref{fig:ptrfbg}.
\begin{figure}[H]
\begin{center}
\begin{tikzpicture}[xscale=1.2,yscale=.9]
\draw (-1.8cm,-.7cm) -- (0cm,0cm) -- (1.8cm,-.7cm) -- (3.6cm,0cm) -- (5.4cm,-.7cm) (0cm,0cm) -- (0cm, 1.8cm) (3.6cm,0cm) -- (3.6cm,1.8cm) (-.1cm,1.6cm) -- (0cm,1.8cm) -- (.1cm,1.6cm) (3.5cm,1.6cm) -- (3.6cm,1.8cm) -- (3.7cm,1.6cm);
\filldraw[black] (-1.8cm,-.7cm) circle (2pt) 
(0cm,0cm) circle (2pt) 
(1.8cm,-.7cm) circle (2pt) 
(3.6cm,0cm) circle (2pt) 
(5.4cm,-.7cm) circle (2pt);
\draw (-1.5cm,-1.1cm) node {$(1)$} 
(-1.8cm,-.3cm) node {$-2$} 
(0cm,-.5cm) node {$(2)$} 
(.45cm,.2cm) node {$-1$} 
(.4cm,1.8cm) node {$(1)$} 
(1.8cm,-1.05cm) node {$(0)$} 
(1.8cm,-.35cm) node {$-5$} 
(3.6cm,-.5cm) node {$(-2)$} 
(4.05cm,.2cm) node {$-1$} 
(4.15cm,1.8cm) node {$(-1)$} 
(5cm,-1.1cm) node {$(-1)$} 
(5.4cm,-.35cm) node {$-2$};
\end{tikzpicture}
\end{center}
\caption{Plumbing tree $\Gamma_{\fbg}$ for $\fbgxy=(x^2+y^3)\overline{(x^3+y^2)}$.}
\label{fig:ptrfbg}
\end{figure}

After the change of orientation of the Seifert fibres as we made above, we obtain the graph in Figure~\ref{fig:choptrfbg}.
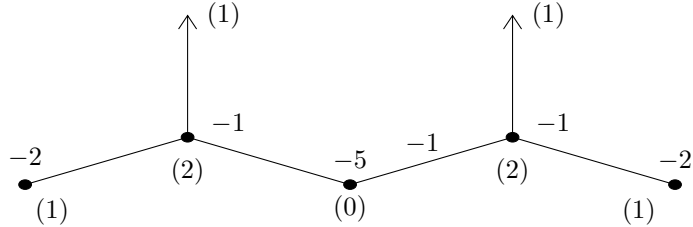
\begin{figure}[H]
\begin{center}
\begin{tikzpicture}[xscale=1.2,yscale=.9]
\draw (-1.8cm,-.7cm) -- (0cm,0cm) -- (1.8cm,-.7cm) -- (3.6cm,0cm) -- (5.4cm,-.7cm) (0cm,0cm) -- (0cm, 1.8cm) (3.6cm,0cm) -- (3.6cm,1.8cm) (-.1cm,1.6cm) -- (0cm,1.8cm) -- (.1cm,1.6cm) (3.5cm,1.6cm) -- (3.6cm,1.8cm) -- (3.7cm,1.6cm);
\filldraw[black] (-1.8cm,-.7cm) circle (2pt) 
(0cm,0cm) circle (2pt) 
(1.8cm,-.7cm) circle (2pt) 
(3.6cm,0cm) circle (2pt) 
(5.4cm,-.7cm) circle (2pt);
\draw (-1.5cm,-1.1cm) node {$(1)$} 
(-1.8cm,-.3cm) node {$-2$} 
(0cm,-.5cm) node {$(2)$} 
(.45cm,.2cm) node {$-1$} 
(.4cm,1.8cm) node {$(1)$} 
(1.8cm,-1.05cm) node {$(0)$} 
(1.8cm,-.35cm) node {$-5$}
(2.6cm, -.1cm) node {$-1$} 
(3.6cm,-.5cm) node {$(2)$} 
(4.05cm,.2cm) node {$-1$} 
(4cm,1.8cm) node {$(1)$} 
(5cm,-1.1cm) node {$(1)$} 
(5.4cm,-.35cm) node {$-2$};
\end{tikzpicture}
\end{center}
\caption{Plumbing tree obtained after the change of orientation in the plumbing tree $\Gamma_{\fbg}$ for $\fbgxy=(x^2+y^3)\overline{(x^3+y^2)}$.}
\label{fig:choptrfbg}
\end{figure}

By Theorem~\ref{mero}, the Nielsen graph $\Ng{h}$ of the monodromy $h$ of the Milnor fibration of $\fbg$ is as in Figure~\ref{fig:ngrfbg}.
\begin{figure}[H]
\begin{center}
\begin{tikzpicture}[xscale=1.1,yscale=.85]
\filldraw[black] (0cm,0cm) circle (2.8pt) (5cm,0cm) circle (2.8pt) (-2cm,-1.4cm) circle (1.6pt) (7cm,-1.4cm) circle (1.6pt);
\draw (-1.95cm,1.365cm) -- (0cm,0cm) --  (5cm,0cm) -- (6.95cm,1.365cm) (-2cm,-1.4cm) -- (0cm,0cm) (-2cm,1.4cm) circle (1.6pt) (7cm,-1.4cm) -- (5cm,0cm) (7cm,1.4cm) circle (1.6pt);
\draw (1cm,.3cm) node {$(1,1)$} 
(0cm,-.3cm) node {$1$} 
(-.9cm,0cm) node {$[2,0]$} 
(4cm,.3cm) node {$(1,1)$} 
(5cm,-.3cm) node {$1$} 
(5.9cm,0cm) node {$[2,0]$} 
(-.4cm,.9cm) node {$(2,1)$} 
(-.5cm,-.9cm) node {$(2,1)$}
(5.4cm,.9cm) node {$(2,1)$} 
(5.5cm,-.9cm) node {$(2,1)$} 
(2.5cm,-.4cm) node[scale=1.5] {$\frac{5}{2}$} 
(-2.5cm,1.4cm) node[scale=1.5] {-$\frac{1}{2}$} 
(7.5cm,1.4cm) node[scale=1.5] {-$\frac{1}{2}$}; 
\end{tikzpicture}
\end{center}
\caption{Nielsen graph $\Ng{h}$ for $\fbgxy=(x^2+y^3)\overline{(x^3+y^2)}$.}
\label{fig:ngrfbg}
\end{figure}
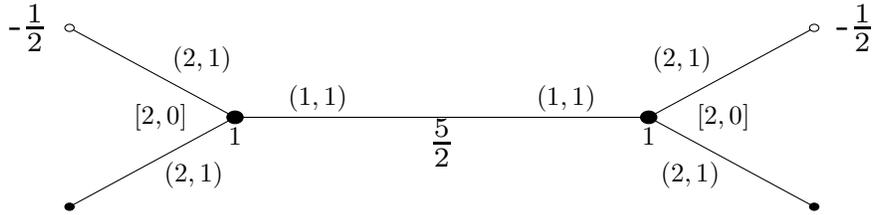

Note that, as we said, the twist is positive as a consequence of the change of sign of $\epsilon$ in the formula \eqref{f3}.

Now, the Nielsen graph $\Ng{\mtp{h}{2}}$ is given in Figure~\ref{fig:ngrfbgr2}.
\begin{figure}[H]
\begin{center}
\begin{tikzpicture}[xscale=1.1,yscale=.9]
\draw (0cm,0cm) .. controls (1cm,.8cm) and (3cm,.8cm) .. (4cm,0cm)
      node [near start, sloped, above] {$(1,1) \quad \ $} 
      node [near end, sloped, above] {$\quad \ (1,1)$}
      node [midway, above] {$5$}
      (0cm,0cm) .. controls (1cm,-.8cm) and (3cm,-.8cm) .. (4cm,0cm)
      node [near start, sloped, below] {$(1,1) \quad \ $} 
      node [near end, sloped, below] {$\quad \ (1,1)$}
      node [midway, below] {$5$};
\filldraw[black] (0cm,0cm) circle (2.8pt) (4cm,0cm) circle (2.8pt);
\draw (-1.95cm,1.365cm) -- (0cm,0cm) (4cm,0cm) -- (5.95cm,1.365cm) (-2cm,1.4cm) circle (1.6pt) (6cm,1.4cm) circle (1.6pt);
\draw 
(-1.45cm,.5cm) node {$(1,-1)$} 
(5.45cm,.5cm) node {$(1,-1)$}
(.6cm,0cm) node {$[1,0]$}
(3.4cm,0cm) node {$[1,0]$}
(-2.3cm,1.4cm) node {$-1$}
(6.3cm,1.4cm) node {$-1$}
(0cm,-.3cm) node {$1$}
(4cm,-.3cm) node {$1$};
\end{tikzpicture}
\end{center}
\caption{Nielsen graph $\Ng{h^2}$ of $h^2$ with $h$ the monodromy of the Milnor fibration of $\fbgxy=(x^2+y^3)\overline{(x^3+y^2)}$.}
\label{fig:ngrfbgr2}
\end{figure}
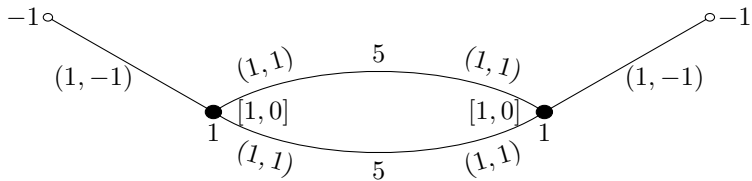

By Theorem~\ref{mero}, the graph $\W(\LF,\Lp)$ is the graph shown in Figure~\ref{fig:wg23322}.
\begin{figure}[H]
\begin{center}
\begin{tikzpicture}[xscale=1.1,yscale=.9]
\draw (0cm,0cm) .. controls (1cm,.8cm) and (3cm,.8cm) .. (4cm,0cm)
      node [midway, above] {$(-1,5,4)$}
      (0cm,0cm) .. controls (1cm,-.8cm) and (3cm,-.8cm) .. (4cm,0cm)
      node [midway, below] {$(-1,5,4)$};
\filldraw[black] (0cm,0cm) circle (2.8pt) (4cm,0cm) circle (2.8pt);
\draw[xshift=4cm] (0cm,0cm) -- ++(2cm,1.4cm) -- +(185:5pt) (0cm,0cm) ++(2cm,1.4cm) -- +(245:5pt);
\draw (0cm,0cm) -- ++(-2cm,1.4cm) -- +(290:5pt) (0cm,0cm) ++(-2cm,1.4cm) -- +(350:5pt);
\draw 
(-1.45cm,.5cm) node {$(1,0)$} 
(5.45cm,.5cm) node {$(1,0)$}
(0cm,-.28cm) node {$1$}
(4cm,-.28cm) node {$1$};
\end{tikzpicture}
\end{center}
\caption{Graph $\W(\LF,\Lp)$ for $F(x,y,z) = (x^2+y^3)\overline{(x^3+y^2)} + z^2$.}
\label{fig:wg23322}
\end{figure}
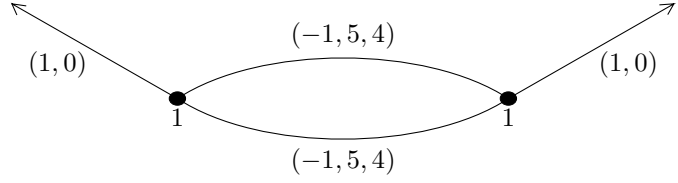
From the graph $\W(\LF,\Lp)$, the corresponding plumbing tree $\Gamma$ can be computed (see Figure~\ref{fig:pg23322}).
\begin{figure}[H]
\begin{center}
\begin{tikzpicture}[xscale=1.3,yscale=1]
\draw (0cm,0cm) .. controls (1cm,.8cm) and (3cm,.8cm) .. (4cm,0cm)
      (0cm,0cm) .. controls (1cm,-.8cm) and (3cm,-.8cm) .. (4cm,0cm);
\filldraw[black] (0cm,0cm) circle (2.5pt) (4cm,0cm) circle (2.5pt) (2cm,.6cm) circle (2.5pt) (2cm,-.6cm) circle (2.5pt);
\draw (2cm,.9cm) node {$-5$}
      (2cm,-.9cm) node {$-5$}
      (-.4cm,0cm) node {$-1$}
      (4.4cm,0cm) node {$-1$};
\end{tikzpicture}
\end{center}
\caption{Plumbing tree $\Gamma$ corresponding to graph $\W(\LF,\Lp)$ for $F(x,y,z) = (x^2+y^3)\overline{(x^3+y^2)} + z^2$.}
\label{fig:pg23322}
\end{figure}
As in the previous example, $\Gamma$ represents the exceptional divisor in the interior of the manifold $V_{\Gamma}$; \ie it is a graph resolution corresponding to a resolution $\widetilde{V}$ of a normal surface singularity $(V,0)$. Then, blowing down the two vertices with weights $-1$, one obtains the graph in Figure~\ref{fig:bpg23322}.
\begin{figure}[H]
\begin{center}
\begin{tikzpicture}[yscale=.7,xscale=.8,rotate=90]
\draw (0cm,0cm) .. controls (1cm,.8cm) and (3cm,.8cm) .. (4cm,0cm)
      (0cm,0cm) .. controls (1cm,-.8cm) and (3cm,-.8cm) .. (4cm,0cm);
\filldraw[black] (0cm,0cm) circle (2.5pt) (4cm,0cm) circle (2.5pt);
\draw (-.4cm,0cm) node {$-3$}
      (4.4cm,0cm) node {$-3$};
\end{tikzpicture}
\end{center}
\caption{Plumbing tree $\Gamma$ corresponding to graph $\W(\LF,\Lp)$ for $F(x,y,z) = (x^2+y^3)\overline{(x^3+y^2)} + z^2$.}
\label{fig:bpg23322}
\end{figure}
This graph is also the plumbing tree for $L_G$, where $G \colon (\C^3,0) \to (\C,0)$ is the germ
\begin{equation*}
G(x,y,z) = (x^2+y^3)(x^3+y^2) + z^2
\end{equation*}
as is stated in \cite[\S~6, Examples]{MR1709489}. Then the link $L_F$ is realisable by an holomorphic function from $\C^3$ to $\C$.

Now we proceed to see if the open book fibrations given by $F$ and $G$ are equivalent. In order to see if the  Milnor fibre $\F_F$ is diffeomorphic to the Milnor fibre $\F_G$, we compute the genus of $\F_F$ and the genus of $\F_G$.

By the decorated plumbing tree given in Figure~\ref{fig:choptrfbg}, the Milnor fibre $(\F_{\fbg})_i$ is the $m_i$-covering of $V_i$ (see Proposition~\ref{order}), then
\begin{equation*}
\chi ((\F_{\fbg})_i) = m_i \chi (V_i) \ .
\end{equation*}
As $V_i$ is a cylinder or a disc for the vertices $v_i$ with valence $2$ and $1$ respectively, our principal interest are the nodes (vertices with valence $\geq 3$); let $v_i$ be a node in the graph in Figure~\ref{fig:choptrfbg}, then
\begin{equation*}
\chi ((\F_{\fbg})_i) = 2 \chi (V_i) \ = 2 (-1) = -2.
\end{equation*}
Then, the genus of $\F_i$ is $0$. ``Gluing'' the pieces $(\F_{\fbg})_i$ for all $i$, we obtain a surface of genus $1$ with two boundary components.

Analogously, for the Milnor fibre $(\F_{fg})_i$ we have
\begin{equation*}
\chi ((\F_{fg})_i) = 10 \chi (V_i) \ = 10 (-1) = -10.
\end{equation*}
Then the genus of $(\F_{fg})_i$ is $2$. ``Gluing'' the pieces $(\F_{fg})_i$ for all $i$, we obtain a surface of genus $5$ with two boundary components.

Thus, the join of $r$ points with $\F_{\fbg}$ cannot be the same as the join of $r$ points with $\F_{fg}$ and the open book decompositions given by the Milnor fibrations of $F$ and $G$ are not equivalent.

However remains open the question if the link $\LF$ is homeomorphic to the link of other complex singularity with equivalent Milnor fibration to $F$.
\end{exa}

\begin{exa}
Let $f \colon (\C^2,0) \to (\C,0)$ be the complex analytic germ defined by $f(x,y)=x^3+y^5$ and let $g \colon (\C^2,0) \to (\C,0)$ be the complex analytic germ defined by $g(x,y)=x^7+y^2$.

\subsubsection*{First step:} 
The resolution graph $A(fg)$, where the multiplicity $m_i = m^f_i+m^g_i$ at the vertex $i$ appears as $\ensuremath{\binom{m^f_i}{m^g_i}}$ is given by
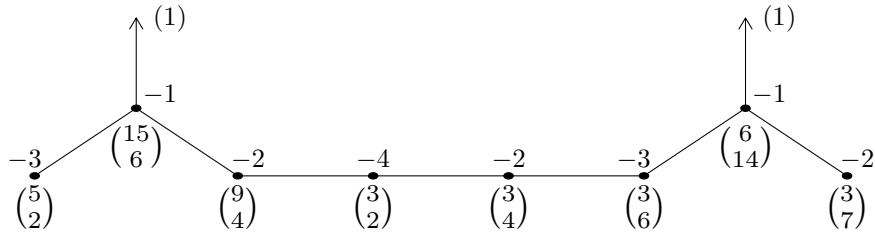
\begin{center}
\begin{tikzpicture}[xscale=.45,yscale=.3]
\draw[black] (0cm,0cm) -- (3cm,3cm) -- (6cm,0cm) -- (18cm,0cm) -- (21cm,3cm) -- (24cm,0cm) (3cm,3cm) -- (3cm,7cm) (21cm,3cm) -- (21cm,7cm) (2.8cm,6.5cm) -- (3cm,7cm) -- (3.2cm,6.5cm) (20.8cm,6.5cm) -- (21cm,7cm) -- (21.2cm,6.5cm);
\filldraw (0cm,0cm) circle (4pt) (3cm,3cm) circle (4pt) (6cm,0cm) circle (4pt) (10cm,0cm) circle (4pt) (14cm,0cm) circle (4pt) (18cm,0cm) circle (4pt) (21cm,3cm) circle (4pt) (24cm,0cm) circle (4pt);
\draw (-.3cm,.7cm) node {$-3$}
(3.7cm,3.6cm) node {$-1$}
(6.3cm,.7cm) node {$-2$}
(10cm,.7cm) node {$-4$}
(14cm,.7cm) node {$-2$}
(17.7cm,.7cm) node {$-3$}
(21.7,3.6cm) node {$-1$}
(24.3cm,.7cm) node {$-2$};
\draw (0cm,-1.4cm) node [scale=1.3] {$\binom{5}{2}$}
(3cm,1.3cm) node [scale=1.3] {$\binom{15}{6}$}
(6cm,-1.4cm) node [scale=1.3] {$\binom{9}{4}$}
(10cm,-1.4cm) node [scale=1.3] {$\binom{3}{2}$}
(14cm,-1.4cm) node [scale=1.3] {$\binom{3}{4}$}
(18cm,-1.4cm) node [scale=1.3] {$\binom{3}{6}$}
(21cm,1.3cm) node [scale=1.3] {$\binom{6}{14}$}
(24cm,-1.4cm) node [scale=1.3] {$\binom{3}{7}$}
(4cm,7cm) node {$(1)$}
(22cm,7cm) node {$(1)$};
\end{tikzpicture}
\end{center}
Then, the plumbing tree $\Gamma_{\fbg}$ is given by
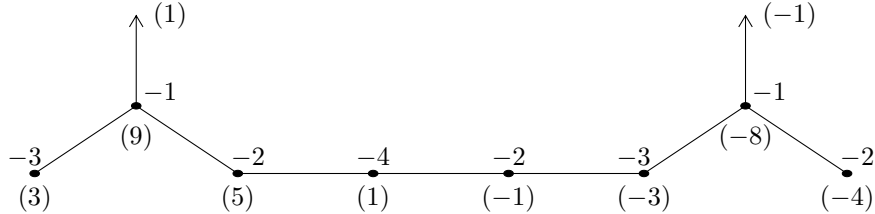
\begin{center}
\begin{tikzpicture}[xscale=.45,yscale=.3]
\draw[black] (0cm,0cm) -- (3cm,3cm) -- (6cm,0cm) -- (18cm,0cm) -- (21cm,3cm) -- (24cm,0cm) (3cm,3cm) -- (3cm,7cm) (21cm,3cm) -- (21cm,7cm) (2.8cm,6.5cm) -- (3cm,7cm) -- (3.2cm,6.5cm) (20.8cm,6.5cm) -- (21cm,7cm) -- (21.2cm,6.5cm);
\filldraw (0cm,0cm) circle (4pt) (3cm,3cm) circle (4pt) (6cm,0cm) circle (4pt) (10cm,0cm) circle (4pt) (14cm,0cm) circle (4pt) (18cm,0cm) circle (4pt) (21cm,3cm) circle (4pt) (24cm,0cm) circle (4pt);
\draw (-.3cm,.7cm) node {$-3$}
(3.7cm,3.6cm) node {$-1$}
(6.3cm,.7cm) node {$-2$}
(10cm,.7cm) node {$-4$}
(14cm,.7cm) node {$-2$}
(17.7cm,.7cm) node {$-3$}
(21.7,3.6cm) node {$-1$}
(24.3cm,.7cm) node {$-2$};
\draw (0cm,-1.1cm) node {$(3)$}
(3cm,1.6cm) node {$(9)$}
(6cm,-1.1cm) node {$(5)$}
(10cm,-1.1cm) node {$(1)$}
(14cm,-1.1cm) node {$(-1)$}
(18cm,-1.1cm) node {$(-3)$}
(21cm,1.6cm) node {$(-8)$}
(24cm,-1.1cm) node {$(-4)$}
(4cm,7cm) node {$(1)$}
(22.3cm,7cm) node {$(-1)$};
\end{tikzpicture}
\end{center}
After the change of orientation in the Seifert fibres, we obtain the following graph:
\begin{figure}[H]
\begin{center}
\begin{tikzpicture}[xscale=.45,yscale=.3]
\draw[black] (0cm,0cm) -- (3cm,3cm) -- (6cm,0cm) -- (18cm,0cm) -- (21cm,3cm) -- (24cm,0cm) (3cm,3cm) -- (3cm,7cm) (21cm,3cm) -- (21cm,7cm) (2.8cm,6.5cm) -- (3cm,7cm) -- (3.2cm,6.5cm) (20.8cm,6.5cm) -- (21cm,7cm) -- (21.2cm,6.5cm);
\filldraw (0cm,0cm) circle (4pt) (3cm,3cm) circle (4pt) (6cm,0cm) circle (4pt) (10cm,0cm) circle (4pt) (14cm,0cm) circle (4pt) (18cm,0cm) circle (4pt) (21cm,3cm) circle (4pt) (24cm,0cm) circle (4pt);
\draw (-.3cm,.7cm) node {$-3$}
(3.7cm,3.6cm) node {$-1$}
(6.3cm,.7cm) node {$-2$}
(10cm,.7cm) node {$-4$}
(14cm,.7cm) node {$-2$}
(17.7cm,.7cm) node {$-3$}
(21.7,3.6cm) node {$-1$}
(24.3cm,.7cm) node {$-2$};
\draw (0cm,-1.1cm) node {$(3)$}
(3cm,1.6cm) node {$(9)$}
(6cm,-1.1cm) node {$(5)$}
(10cm,-1.1cm) node {$(1)$}
(11.8cm,-.7cm) node {$-1$}
(14cm,-1.1cm) node {$(1)$}
(18cm,-1.1cm) node {$(3)$}
(21cm,1.6cm) node {$(8)$}
(24cm,-1.1cm) node {$(4)$}
(4cm,7cm) node {$(1)$}
(22cm,7cm) node {$(1)$};
\end{tikzpicture}
\end{center}
\caption{Plumbing tree obtained after change of orientation.}
\label{rgfbg}
\end{figure}
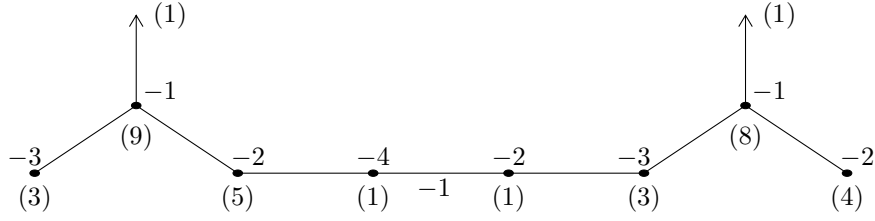

\subsubsection*{Second step:}
From the last graph, we compute the Nielsen graph $\Ng{h}$ of the monodromy $h$ of the Milnor fibration $\mpfg$.
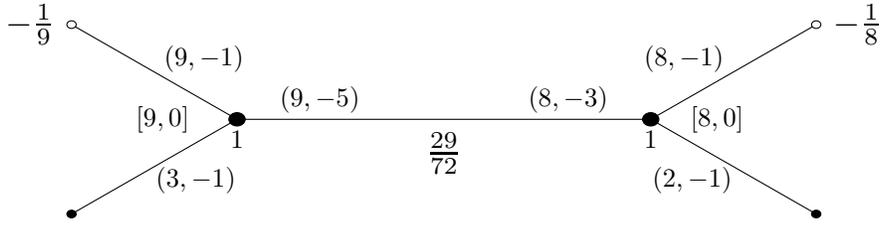
\begin{center}
\begin{tikzpicture}[xscale=1.1,yscale=.9]
\filldraw[black] (0cm,0cm) circle (2.8pt) (5cm,0cm) circle (2.8pt) (-2cm,-1.4cm) circle (1.6pt) (7cm,-1.4cm) circle (1.6pt);
\draw (-1.95cm,1.365cm) -- (0cm,0cm) --  (5cm,0cm) -- (6.95cm,1.365cm) (-2cm,-1.4cm) -- (0cm,0cm) (-2cm,1.4cm) circle (1.6pt) (7cm,-1.4cm) -- (5cm,0cm) (7cm,1.4cm) circle (1.6pt);
\draw (1cm,.3cm) node {$(9,-5)$}
(0cm,-.3cm) node {$1$}
(-.9cm,0cm) node {$[9,0]$}
(4cm,.3cm) node {$(8,-3)$}
(5cm,-.3cm) node {$1$}
(5.8cm,0cm) node {$[8,0]$}
(-.4cm,.9cm) node {$(9,-1)$}
(-.5cm,-.9cm) node {$(3,-1)$}
(5.4cm,.9cm) node {$(8,-1)$}
(5.5cm,-.9cm) node {$(2,-1)$}
(2.5cm,-.5cm) node[scale=1.3] {$\frac{29}{72}$}
(-2.5cm,1.4cm) node [scale=1.3] {$-\frac{1}{9}$}
(7.5cm,1.4cm) node [scale=1.3] {$-\frac{1}{8}$}; 
\end{tikzpicture}
\end{center}
where the twist $\frac{29}{72}$ is computed by applying equation~\eqref{f3}:
\begin{equation*}
t= \frac{\alpha}{m_j \lambda} = \frac{29}{(8)(9)} = \frac{29}{72} \ .
\end{equation*}

\subsubsection*{Third step:}
Let $F \colon (\C^3,0) \to (\C,0)$ be the real analytic germ defined by
\begin{equation*}
F(x,y,z) = (x^3+y^5)\overline{(x^7+y^2)} + z^5 \ .
\end{equation*}
The link $\LF$ has an open book fibration with binding $\Lp$ and monodromy $\mtp{h}{5}$. We compute the Nielsen graph $\Ng{\mtp{h}{5}}$ from the graph $\Ng{h}$:
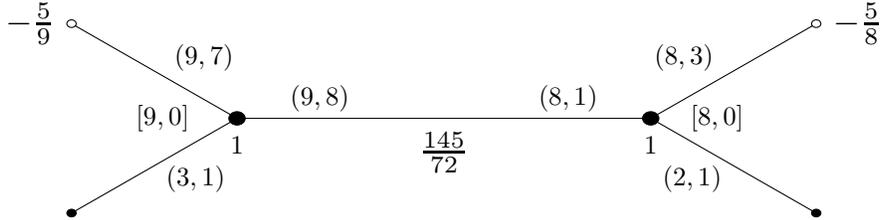
\begin{center}
\begin{tikzpicture}[xscale=1.1,yscale=.9]
\filldraw[black] (0cm,0cm) circle (2.8pt) (5cm,0cm) circle (2.8pt) (-2cm,-1.4cm) circle (1.6pt) (7cm,-1.4cm) circle (1.6pt);
\draw (-1.95cm,1.365cm) -- (0cm,0cm) --  (5cm,0cm) -- (6.95cm,1.365cm) (-2cm,-1.4cm) -- (0cm,0cm) (-2cm,1.4cm) circle (1.6pt) (7cm,-1.4cm) -- (5cm,0cm) (7cm,1.4cm) circle (1.6pt);
\draw (1cm,.3cm) node {$(9,8)$}
(0cm,-.4cm) node {$1$}
(-.9cm,0cm) node {$[9,0]$}
(4cm,.3cm) node {$(8,1)$}
(5cm,-.4cm) node {$1$}
(5.8cm,0cm) node {$[8,0]$}
(-.4cm,.9cm) node {$(9,7)$}
(-.5cm,-.9cm) node {$(3,1)$}
(5.4cm,.9cm) node {$(8,3)$}
(5.5cm,-.9cm) node {$(2,1)$}
(2.5cm,-.5cm) node[scale=1.3] {$\frac{145}{72}$}
(-2.5cm,1.4cm) node [scale=1.3] {$-\frac{5}{9}$}
(7.5cm,1.4cm) node [scale=1.3] {$-\frac{5}{8}$}; 
\end{tikzpicture}
\end{center}

\subsubsection*{Fourth step:}

Applying Theorem~\ref{mero}, we compute the graph $\W(\LF,\Lp)$ from the Nielsen graph $\Ng{\mtp{h}{5}}$:
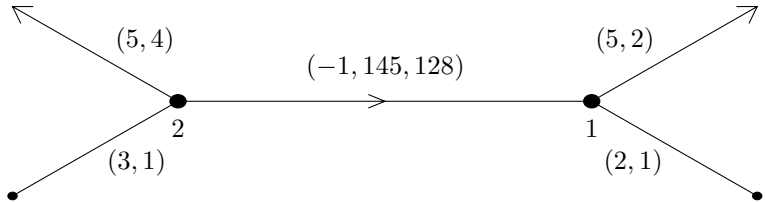
\begin{center}
\begin{tikzpicture}[xscale=1.1,yscale=.9]
\filldraw[black] (0cm,0cm) circle (2.8pt) (5cm,0cm) circle (2.8pt) (-2cm,-1.4cm) circle (1.6pt) (7cm,-1.4cm) circle (1.6pt);
\draw (-2cm,1.4cm) -- (0cm,0cm) --  (5cm,0cm) -- (7cm,1.4cm) (-2cm,-1.4cm) -- (0cm,0cm) (-1.9cm,1.15cm) -- (-2cm,1.4cm) -- (-1.75cm,1.4cm) (7cm,-1.4cm) -- (5cm,0cm) (6.9cm,1.15cm) -- (7cm,1.4cm) -- (6.75cm,1.4cm) (2.3cm,.1cm) -- (2.5cm,0cm) -- (2.3cm,-.1cm);
\draw (2.5cm,.5cm) node {$(-1,145,128)$}
(0cm,-.4cm) node {$2$}
(5cm,-.4cm) node {$1$}
(-.4cm,.9cm) node {$(5,4)$}
(-.5cm,-.9cm) node {$(3,1)$}
(5.4cm,.9cm) node {$(5,2)$}
(5.5cm,-.9cm) node {$(2,1)$};
\end{tikzpicture}
\end{center}

\subsubsection*{Fifth step:}
From the graph $\W(\LF,\Lp)$ we can compute the corresponding plumbing tree. As we have the following equations:
\begin{gather*}
\frac{5}{5-4} = 5 = [5] \ , \quad \ \frac{3}{3-1} = 2-\frac{1}{2} = [2,2] \ , \quad \ \frac{5}{5-2} = 2-\frac{1}{3} = [2,3] \\
\frac{2}{2-1} = 2 = [2] \ , \quad \ \frac{145}{128} = [9,3,2,2,2,2,2,2,2] \ ,
\end{gather*}
then the plumbing tree $\Gamma$ is given by Figure~\ref{largpt},
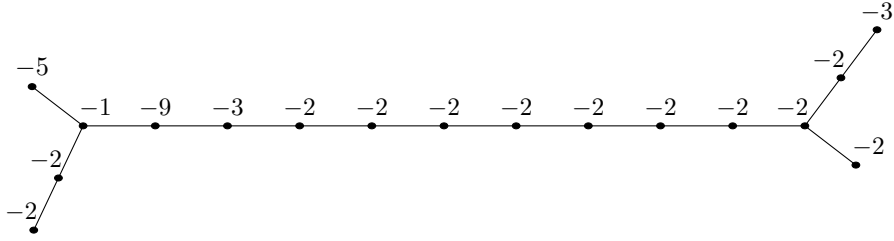
\begin{figure}[H]
\begin{center}
\begin{tikzpicture}[xscale=.91,yscale=.7]
\draw (135:30pt) -- (0pt,0pt) -- (300pt,0pt) -- +(60:60pt) (300pt,0pt) -- +(-45:30pt) (0pt,0pt) -- +(250:60pt);
\filldraw (135:30pt) circle (1.6pt) node [above] {$-5$} 
(0pt,0pt) circle (1.6pt) node [above,scale=.95] {$\quad -1$}
+(250:30pt) circle (1.6pt) node [above,scale=.95] {$-2 \quad$}
+(250:60pt) circle (1.6pt) node [above,scale=.95] {$-2 \quad$}
(30pt,0pt) circle (1.6pt) node [above,scale=.95] {$-9$}
(60pt,0pt) circle (1.6pt) node [above,scale=.95] {$-3$}
(90pt,0pt) circle (1.6pt) node [above,scale=.95] {$-2$}
(120pt,0pt) circle (1.6pt) node [above,scale=.95] {$-2$}
(150pt,0pt) circle (1.6pt) node [above,scale=.95] {$-2$}
(180pt,0pt) circle (1.6pt) node [above,scale=.95] {$-2$}
(210pt,0pt) circle (1.6pt) node [above,scale=.95] {$-2$}
(240pt,0pt) circle (1.6pt) node [above,scale=.95] {$-2$}
(270pt,0pt) circle (1.6pt) node [above,scale=.95] {$-2$}
(300pt,0pt) circle (1.6pt) node [above,scale=.95] {$-2 \quad$}
+(60:30pt) circle (1.6pt) node [above,scale=.95] {$-2 \quad$}
+(60:60pt) circle (1.6pt) node [above,scale=.95] {$-3$}
+(-45:30pt) circle (1.6pt) node [above,scale=.95] {$\quad -2$};
\end{tikzpicture}
\end{center}
\caption{The link $\LF$ as boundary of a plumbing.}
\label{largpt}
\end{figure}
\noindent where $\LF \cong \partial P(\Gamma)$ with $P(\Gamma)$ the four-manifold obtained by plumbing $2$-discs bundles according to $\Gamma$.

Now, as before, the plumbing $P(\Gamma)$ contains in its interior an exceptional divisor $E$ as a strong deformation retract, then $E$ can be blown down to a point, and we get a complex surface $V_\Gamma$ with a normal singularity at $0$. We compute the canonical class $K$ of $V_\Gamma$ and we obtain that
\begin{equation*}
K=(-27, -18, -9, -6, -4, -2, -1,\, 0,\, 1,\, 2,\, 3,\, 4,\, 5,\, 6,\, 3,\, 4,\, 2) \ ,
\end{equation*}
then $V_\Gamma$ is numerically Gorenstein. Then there could be a complex analytic germ $G \colon (\C^3,0) \to (\C,0)$ with isolated singularity at the origin such that the link $L_G$ is homeomorphic to the link $L_F$.

Following \cite[Section~5]{MR2922705}, we use an adaptation of the Laufer-Steenbrink formula as an obstruction for the equivalence of the Milnor fibration of $G$ and the Milnor fibration of $F$.

\begin{thm}[{\cite[\S~4, Cor~1]{MR713273}}]\label{cong}
Let $(V,0)$ be a normal Gorenstein complex surface singularity with link $L$. If $(V,0)$ is smoothable, then one has
\begin{equation*}
\chi(\widetilde V) + K^2 \equiv \chi(V^{\#}) \pmod{12}
\end{equation*}
where $V^{\#}$ is a smoothing of $V$, $\widetilde V$ is a good resolution of $V$ and $K$ is the canonical class of $\widetilde V$.
\end{thm}

From Figure~\ref{rgfbg}, we obtain that the Milnor fibre of $\fbg$, $\F$, has genus $5$ and two boundary components, then
\begin{equation*}
\chi(\F)=2-2g-2=-10 \ .
\end{equation*}
Also $\F$ is homotopically equivalent to $\bigvee_{i=1}^k \Sp{1}_i$, where $k= 1 - \chi(\F)$. Let $\F_F$ be the Milnor fibre of $F$, then by Theorem~\ref{jointh} we have that
\begin{equation}\label{ecmf}
\chi(\F_F)= \chi \Biggl(\bigvee_{j=1}^{(3-1)k} \Sp{2}_j \Biggr) = 1 + (3-1)\left(1-(-10)\right)=23 \ .
\end{equation}

On the other hand, Figure~\ref{largpt} has $17$ vertices and $16$ edges, then
\begin{equation}\label{ecrg}
\chi(\widetilde V) = 2(17) - 16 = 18
\end{equation}
and
\begin{equation}\label{aicc}
K^2=K^T A K = -33 \ ,
\end{equation}
where $A$ is the intersection matrix of the plumbing in Figure~\ref{largpt}.

Combining equations~\eqref{ecmf}, \eqref{ecrg} and \eqref{aicc}, we obtain that
\begin{equation*}
\chi(\F_F) = 23 \not\equiv -15 = \chi(\widetilde V) + K^2 \pmod{12} \ .
\end{equation*}
Then, given a complex analytic germ $G \colon (\C^3,0) \to (\C,0)$ with isolated singularity at the origin such that the link $L_G$ is homeomorphic to the link $L_F$, the Milnor fibration of $G$ is not equivalent to the Milnor fibration of $F$.
\end{exa}

\begin{rmk}
In Example~\ref{exfeo}, Laufer-Steenbrink formula (which gives Theorem~\ref{cong}) does not work as an obstruction as in the last example. In that case $\chi(\F_F)$ and $\chi(\widetilde V) + K^2$ are congruent modulo $12$.
\end{rmk}

It remains to look for other kind of obstruction in order to compare the Milnor fibration of the function $F$ of Example~\ref{exfeo} with the Milnor fibration of a complex analytic germ $G \colon (\C^3,0) \to (\C,0)$.

\section*{Acknowledgements}
I am most grateful to Anne Pichon and Jos\'e Seade for their supervision and comments on this work. I also want to thank Professor Walter Neumann and Patrick Popescu-Pampu for all their suggestions which improved greatly the present article, and Jos\'e Luis Cisneros-Molina for helpful conversations and remarks.


\begin{thebibliography}{10}
\bibitem{MR2922705}
H.~Aguilar-Cabrera.
\newblock New open-book decompositions in singularity theory.
\newblock {\em Geom. Dedicata}, 158:87--108, 2012.

\bibitem{MR1877769}
C.~Ban, L.~J. McEwan, and A.~N{\'e}methi.
\newblock The embedded resolution of {$f(x,y)+z^2 \colon (\mathbb{C}^3,0) \to
  (\mathbb{C},0)$}.
\newblock {\em Studia Sci. Math. Hungar.}, 38:51--71, 2001.

\bibitem{NC96}
N.~Chaves.
\newblock {\em Vari\'et\'es graph\'ees fibr\'ees sur le cercle et
  diff\'eomorphismes quasi-finis de surfaces}.
\newblock PhD thesis, Universit\'e de Gen\'eve, 1996.

\bibitem{MR2647448}
J.~L. Cisneros-Molina, J.~Seade, and J.~Snoussi.
\newblock Milnor fibrations and {$d$}-regularity for real analytic
  singularities.
\newblock {\em Internat. J. Math.}, 21(4):419--434, 2010.

\bibitem{DBM}
P.~Du~Bois and F.~Michel.
\newblock The integral {S}eifert form does not determine the topology of plane
  curve germs.
\newblock {\em Journal of Algebraic Geometry}, 3:1--38, 1994.

\bibitem{durfee:neighalgs}
A.~H. Durfee.
\newblock Neighborhoods of algebraic sets.
\newblock {\em Trans. Amer. Math. Soc.}, 276(2):517--530, abril 1983.

\bibitem{EN85}
D.~Eisenbud and W.~D. Neumann.
\newblock {\em Three dimensional link theory and invariants of plane curve
  germs}.
\newblock Number 110 in Annals of Math. Studies. Princeton University Press,
  1985.

\bibitem{MR656605}
F.~Gonz{\'a}lez Acu\~na and J.~M. Montesinos.
\newblock Embedding knots in trivial knots.
\newblock {\em Bull. London Math. Soc.}, 14(3):238--240, 1982.

\bibitem{MR0488073}
L.~H. Kauffman and W.~D. Neumann.
\newblock Products of knots, branched fibrations and sums of singularities.
\newblock {\em Topology}, 16(4):369--393, 1977.

\bibitem{MR2134278}
R.~Mendris and A.~N{\'e}methi.
\newblock The link of {$\{f(x,y)+z^n=0\}$} and {Z}ariski's conjecture.
\newblock {\em Compos. Math.}, 141(2):502--524, 2005.

\bibitem{milnor:singular}
J.~Milnor.
\newblock {\em Singular points of complex hypersurfaces}.
\newblock Annals of Mathematics Studies, No. 61. Princeton University Press,
  Princeton, N.J., 1968.

\bibitem{MR915761}
J.~M. Montesinos.
\newblock {\em Classical tessellations and three-manifolds}.
\newblock Universitext. Springer-Verlag, Berlin, 1987.

\bibitem{MR1669948}
A.~N{\'e}methi.
\newblock Dedekind sums and the signature of {$f(x,y)+z^N$}.
\newblock {\em Selecta Math. (N.S.)}, 4(2):361--376, 1998.

\bibitem{MR0358797}
W.~D. Neumann.
\newblock Cyclic suspension of knots and periodicity of signature for
  singularities.
\newblock {\em Bull. Amer. Math. Soc.}, 80:977--981, 1974.

\bibitem{Neu:calcplumb}
W.~D. Neumann.
\newblock A calculus for plumbing applied to the topology of complex surface
  singularities and degenerating complex curves.
\newblock {\em Trans. Amer. Math. Soc.}, 268(2):299--344, 1981.

\bibitem{NeuRay:plumb}
W.~D. Neumann and F.~Raymond.
\newblock Seifert manifolds, plumbing, {$\mu $}-invariant and orientation
  reversing maps.
\newblock In {\em Algebraic and geometric topology (Proc. Sympos., Univ.
  California, Santa Barbara, Calif., 1977)}, volume 664 of {\em Lecture Notes
  in Math.}, pages 163--196. Springer, Berlin, 1978.

\bibitem{Nielsen:strk}
J.~Nielsen.
\newblock The structure of periodic surface transformations.
\newblock In {\em Jakob {N}ielsen - {C}ollected {M}athematical {P}apers},
  volume~2. Birkhauser, 1986.

\bibitem{MR0015791}
J.~Nielsen.
\newblock Surface transformation classes of algebraically finite type.
\newblock In {\em Jakob {N}ielsen - {C}ollected {M}athematical {P}apers},
  volume~2. Birkhauser, 1986.

\bibitem{MR1709489}
A.~Pichon.
\newblock Three-dimensional manifolds which are the boundary of a normal
  singularity {$z\sp k-f(x,y)$}.
\newblock {\em Math. Z.}, 231(4):625--654, 1999.

\bibitem{MR1824957}
A.~Pichon.
\newblock Fibrations sur le cercle et surfaces complexes.
\newblock {\em Ann. Inst. Fourier (Grenoble)}, 51(2):337--374, 2001.

\bibitem{MR2115674}
A.~Pichon.
\newblock Real analytic germs {$f\overline g$} and open-book decompositions of
  the 3-sphere.
\newblock {\em Internat. J. Math.}, 16(1):1--12, 2005.

\bibitem{PichSea:barfg}
A.~Pichon and J.~Seade.
\newblock Fibred multilinks and singularities $f\bar{g}$.
\newblock {\em Mathematische Annalen}, Volume 342(Number 3):487--514, 2008.

\bibitem{MR713273}
J.~Seade.
\newblock A cobordism invariant for surface singularities.
\newblock In {\em Singularities, {P}art 2 ({A}rcata, {C}alif., 1981)},
  volume~40 of {\em Proc. Sympos. Pure Math.}, pages 479--484. Amer. Math.
  Soc., Providence, R.I., 1983.
\end{thebibliography}

\end{document}